\documentclass[12pt]{amsart}
\usepackage{amsmath,amsthm,amsfonts,amssymb,latexsym, mathrsfs}
\usepackage{graphicx, xcolor} 

\usepackage{hyperref}

  \title{The fields generated by character values of Linear and Unitary Groups}

\author{Eden Ketchum}
\thanks{ This work was partially supported by the U.S. National Science Foundation, Award No. DMS-2439897}

\usepackage{ marvosym }
\usepackage{amssymb}
\usepackage{amsmath}
\usepackage{mathtools}    
\usepackage{xcolor}   \usepackage{ wasysym }
\usepackage{mathrsfs}
\usepackage{yfonts}
\usepackage{stmaryrd}

\usepackage{amsfonts}
\usepackage{fancyhdr}
\usepackage{comment}
\usepackage[utf8]{inputenc}
\usepackage[a4paper, top=2.5cm, bottom=2.5cm, left=2.2cm, right=2.2cm]{geometry}
\usepackage{pst-node}
\usepackage{tikz-cd} 

\usepackage{blindtext}
\usepackage{amssymb}
\usepackage{xcolor}

\usepackage{amsthm}

\usepackage{commath}
\usepackage{times}
\usepackage{amsmath}
\usepackage{changepage}
\usepackage{amssymb}
\usepackage{graphicx}
\usepackage{chngcntr}
\usepackage{stackengine}
\usepackage{yfonts}
\usepackage{ marvosym }

\usepackage{thmtools}
\usepackage{enumitem,kantlipsum}
\usepackage{thmtools}
\usepackage{tikz-cd}
\usepackage[normalem]{ulem}
\usepackage{cleveref}

\newcommand{\C}{\mathbb{C}}

\newcommand{\Z}{\mathbb{Z}}
\newcommand{\Q}{\mathbb{Q}}

\newcommand{\Stab}{\operatorname{Stab}}

\usepackage{setspace} 

\newtheorem*{theorem*}{Theorem}
\newtheorem{theorem}{Theorem}[section]
\newtheorem{lemma}[theorem]{Lemma}
\newtheorem{corollary}[theorem]{Corollary}
\newtheorem{proposition}[theorem]{Proposition}

\newtheorem{remark}[theorem]{Remark}

\newcommand{\Sl}[2]{\operatorname{SL}}
\newcommand{\SL}{\operatorname{SL}}
\newcommand{\Irr}{\operatorname{Irr}}

\usepackage{ytableau}

\newcommand{\Gl}{\operatorname{GL}}
\newcommand{\GL}{\operatorname{GL}}
\newcommand{\GU}{\operatorname{GU}}
\newcommand{\SU}{\operatorname{SU}}
\newcommand{\E}{\mathcal{E}}
\newcommand{\e}{\epsilon}

\newcommand{\PSl}[2]{\operatorname{PSL}}
\newcommand{\PSL}{\operatorname{PSL}}
\newcommand{\PSU}{\operatorname{PSU}}

\newcommand{\PSp}{\operatorname{PSp}}
\newcommand{\Gal}{\mathrm{Gal}}
\newcommand{\G}{\mathcal{G}}
\newcommand{\PGL}{\operatorname{PGL}}

\begin{document}

\begin{abstract}
     For a finite group $G$, define the field $\Q(G) := \Q(\{\chi(g) : \chi \in \Irr(G) , g \in G\})$. In this article we compute generating sets for the fields $\Q(G)$ when $G$ is a finite general or special linear group, a finite general or special unitary group, or one of the simple groups $\PSL_n(q)$ or $\operatorname{PSU}_n(q)$. We then apply these results to classify all number fields $F$ of degree $2$ or $3$ over $\Q$ such that $F = \Q(S)$ for some non-abelian  simple group $S$. 
\end{abstract} 
\maketitle

\section{Introduction}

The study of the action of Galois automorphisms on the ordinary irreducible characters of finite groups has yielded many exciting  recent results (see for example \cite{RS26}). A less studied object in this area is the field $\Q(G) := \Q(\{\chi(g) : \chi \in \Irr(G) , g \in G\})$ associated to a finite group $G$. There has been work studying finite groups $G$ such that $\Q(G) =\Q$ (see \cite{FS88} and \cite{Tho08}), but more general questions largely remain unanswered. In particular the question of whether or not all abelian number fields are of the form $\Q(G) $ for some finite group $G$ to our knowledge remains open. As far as the author is aware this was first asked in the closing remarks of \cite{FG72}. This question also appears in \cite{Nav23}, a survey of open problems compiled by Navarro.

One way of approaching such problems may be through simple groups. The fields $\Q(A_n)$ for the alternating groups $A_n$ are well understood due to the work of Robinson and Thompson in \cite{RT95}; however, for groups of Lie type the situation is less understood. An example of a result is \cite{DOW23}, in which the authors calculated these fields for the groups $\GL_n(q)$ and $\SL_2(q)$. In the main results of our article we use recent developments in the study of the fields of character values of finite groups of Lie type to continue the work started in \cite{DOW23}.

Before stating our results we introduce some notation. For $n \in \mathbb{Z}^+$ we let $\xi_n$ denote $\exp((2\pi i)/n)$.
If $d\in \mathbb{Z}^+$, $q$ is some fixed prime power,  and $\epsilon \in \{\pm 1\}$, then we  denote
 $$
 \gamma_{\epsilon,q}(d) := \sum_{k=0}^{d-1} \xi_{q^d-\e^d}^{(\epsilon q)^k}.
 $$
 Then, using the notational convention $\GL^{+1}_n(q): = \GL_n(q)$ and $\GL^{-1}_n(q) := \GU_n(q)$,  we obtain the following result, which generalizes \cite[Theorem 1.1]{DOW23}.
\begin{restatable}{theorem}{main}\label{main1}
Let $q$ be a power of a prime, let $\e \in\{\pm 1\}$, and let $n \ge 1$ be an integer. Then,
 $$\Q(\GL^{\epsilon}_n(q)) = \Q(\{\gamma_{\e,q}(d): 1 \le d \le n\}).$$
\end{restatable}

We note that the case where $\epsilon = 1$ does not exactly match the statement of \cite[Theorem 1.1]{DOW23}, as in that theorem the authors also adjoin all the Galois conjugates of the elements $\gamma_{1,q}(d)$ for $1\le d\le n$. However, since $\Q(\gamma_{1,q}(d))$ is contained in the abelian number field $\Q(\xi_{q^d-1})$, it is immediate from the fundamental theorem of Galois theory that $\Q(\gamma_{1,q}(d))$ is Galois over $\Q$  and contains all the Galois conjugates of  $\gamma_{1,q}(d)$. Therefore, we do not include these elements in our generating set. We further note that these fields are exactly the fields of values of the Brauer character associated to the natural representation that these groups, as matrix groups over a field of positive characteristic, come equipped with (see Corollary \ref{brauer} below). 

If we use the same notational convention as above and  denote similarly  $\SL^{+1}_n(q) := \SL_n(q), $ $\SL^{-1}_n(q) := \SU_n(q), $ $\PSL^{+1}_n(q) := \PSL_n(q)$ and $\PSL^{-1}_n(q) := \PSU_n(q)$ we also obtain the following.

\begin{theorem}\label{main}
   Let $q$ be a power of a prime $p$, $3 \le n \in \mathbb{N}^+$, $\e \in \{\pm1\}, a=(q^n- \epsilon^n)/{(q-\epsilon)} $,  and $b = (q-\epsilon,a) $. If $q$ is odd let $\eta \in \{\pm 1\}$ such that $ \eta \equiv p(\text {mod }4)$. Further we write 
   $$X_a := \xi_a^{ }+\xi_a^{\epsilon q}+...+\xi_a^{(\epsilon q)^{n-1}}$$
   and
   $$X_{a/b}  := \xi_{a/b}^{ }+\xi_{a/b}^{\e q}+...+\xi_{a/b}^{(\e q)^{n-1}}.$$ 
  We have the following
     $$\Q(\SL^{\epsilon}_n(q))= \begin{cases} 
      \Q(\GL^{\epsilon}_{n-1}(q))(X_a,\sqrt{\eta q}) & q \text{ odd, } q \text{ not a square, } 2\le n_2\le (q-\epsilon)_2\\
      \Q(\GL^{\epsilon}_{n-1}(q))(X_a) & \text{otherwise} \\
      \end{cases}.$$
      Further, if we have
   $$
(n,q,\e) \notin\{(3,2,-1),(3,7,1),(3,5,-1),(3,4,1),(4,3,1),(4,5,1),(4,3,-1),(5,4,-1)\},
$$
then
       $$\Q(\PSL^{\epsilon}_n(q))= \begin{cases} 
      \Q(\GL^{\epsilon}_{n-1}(q))(X_{a/b},\sqrt{\eta q}) & q \text{ odd, } q \text{ not a square, } 2\le n_2\le (q-\epsilon)_2\\
      \Q(\GL^{\epsilon}_{n-1}(q))(X_{a/b}) & \text{otherwise} .\\
   \end{cases}$$ 
   If $\eta'\equiv q \text{ (mod 4)}$ with $\eta' \in \{\pm1\}$, then $$\Q(\PSL_2(q)) =    \begin{cases} 
      \Q(\xi_{q-1} + \xi_{q-1}^{-1},\xi_{q+1}+\xi_{q+1}^{-1}) & p = 2\\
     \Q(\xi_{(q-1)/2} + \xi_{(q-1)/2}^{-1},\xi_{(q+1)/2}+\xi_{(q+1)/2}^{-1},\sqrt{\eta' q}) & p >2\\
   \end{cases}.$$

 \end{theorem}

We note that the conditions under which we adjoin $\sqrt{\eta q}$ or $\sqrt{\eta' q}$ are exactly the conditions that appear in \cite[Section 4]{Tur01} and \cite[Theorem 6.1]{SV19}, and that these results are relied upon for the proof of Theorem \ref{main}.

 In the final section we use these results to obtain the following:
\begin{theorem}\label{SmallFields}
    Let $F$ be an extension of $\Q$ such that $|F:\Q| \le 3$. Then, $F = \Q(S)$ for some non-abelian simple group $S$, if and only if $F$ is one of $\Q(\sqrt{5})$, $\Q(\sqrt{13})$, $\Q(\sqrt{17})$,  
     $\Q(\sqrt{21})$, $\Q(\sqrt{-3})$, $
 \Q(\sqrt{-7})$, $\Q(\sqrt{-11})$ or 
     $\Q(\sqrt{-15})$. 
\end{theorem}
\textbf{Acknowledgments:} 
This work was completed as part of my Ph.D. study at the University of Denver. I would like to thank and acknowledge my advisor A. A. Schaeffer Fry for her support and guidance. I would also like to thank Gabriel Navarro for his suggestion to consider which degree-$2$ number fields over the rationals occur as character fields of simple groups.

\section{Preliminaries}

\subsection{Notation}
 Throughout, as in the introduction, we denote $\GL^{+1}_n(q): = \GL_n(q)$ and $\GL^{-1}_n(q) := \GU_n(q)$. Similarly we let $\SL^{+1}_n(q) := \SL_n(q) ,\SL^{-1}_n(q) := \SU_n(q),\PSL^{+1}_n(q) := \PSL_n(q)$, and $\PSL^{-1}_n(q) := \PSU_n(q)$. Further, for a fixed positive integer $n$, $\xi_n$ will denote $\exp(2\pi i/n) \in \C$. For a given prime $p$ we fix an isomorphism from $\overline{\mathbb{F}_p}^\times$ to the group of roots of unity of $p'$-order   in $\mathbb{C}^\times$. We then let $\zeta_{n}$ denote the preimage of $\xi_n$ under this isomorphism. The appropriate $p$ will always be clear from context. We also denote $\G := \Gal(\Q^{ab}:\Q)$.

\subsection{Jordan Decomposition of Characters} Here we provide some details of Deligne-Lusztig  theory, which we will use throughout, and we refer the reader to Chapters 1 and 2 of \cite{GM20} for further details. Our general setting for this discussion will be the following. Let $\mathbf{G}$ be a connected reductive algebraic group defined over an algebraically closed field of positive characteristic $p$ and let $F:\mathbf{G}\rightarrow \mathbf{G}$ be a Steinberg endomorphism. We call the fixed points of $\mathbf{G}$ under $F$, denoted $\mathbf{G}^F$, a \textit{finite group of Lie type}. Further, if $\mathbf{G}$ is of simply connected type and the finite group $\mathbf{G}^F/\mathbf{Z}(\mathbf{G}^F)$ is simple, we call $\mathbf{G}^F/\mathbf{Z}(\mathbf{G}^F)$ a \textit{simple group of Lie type}. 

Given a pair $(\mathbf{G},F)$ as above, we obtain a corresponding dual pair $(\mathbf{G}^*,F^*)$ (cf. \cite[Definition 1.5.17]{GM20}). We then obtain a partition of the irreducible characters of $\mathbf{G}^F$ with blocks indexed by $(\mathbf{G}^*)^{F^*}$ conjugacy classes of $p'$-order elements of $(\mathbf{G}^{*})^{F^*}$ (we call such elements \textit{semisimple}). For a given semisimple element $s \in (\mathbf{G}^*)^{F^*}$, we denote the block of this partition indexed by the conjugacy class of $s$ by $\mathcal{E}(\mathbf{G}^F,s)$. These blocks are called \textit{rational Lusztig series}. The property of rational Lusztig series which is most pertinent to our discussion is the following.

\begin{remark}\label{actiononlusztig}
From \cite[Proposition 3.3.15]{GM20}  we see that if $G$ is a finite group of Lie type corresponding to a pair $(\mathbf{G},F)$ as above and $\sigma \in \mathcal{G}$ such that   $\sigma(\xi_{|G|}) = \xi_{|G|}^r$, then for any rational Lusztig series $\mathcal{E}(G,s)$ and any $\chi \in\mathcal{E}(G,s) $, we have $\sigma(\chi) \in \mathcal{E}(G,s^r)$. We therefore obtain an action of $\mathcal{G}$ on conjugacy classes of semisimple elements of $(\mathbf{G}^*)^{F^*}$ given by $\sigma(s) = s^r$ when $\sigma(\mathcal{E}(G,s)) = \mathcal{E}(G,s^r)$. We will use \cite[Proposition 3.3.15]{GM20} as well as this action throughout without reference.
\end{remark}

The characters in the series $\mathcal{E}(\mathbf{G}^F,1)$ are referred to as \textit{unipotent characters}. If $\mathbf{G}$ has a connected center, then, given a rational Lusztig series $\mathcal{E}(\mathbf{G}^F,s)$, there exists a corresponding  \textit{Jordan decomposition map}, which is a bijection $$
J_s:\mathcal{E}(\mathbf{G}^F,s)\rightarrow\mathcal{E}((\mathbf{C}_{\mathbf{G}^*}(s))^{F^*},1)
$$ 
subject to a certain set of criteria. We refer the reader to  \cite[section 2.6]{GM20} for more details (note that there is a more general version of the Jordan decomposition of characters that applies to the case where $\mathbf{G}$ does not have connected center). While these maps are not in general unique, if $\mathbf{G}$ has a connected center, which is true for general linear and general unitary groups for example, the work of Srinivasan  and Vinroot in \cite{SV20} allows us to make a useful choice of Jordan decomposition maps (see also \cite{STV25} for a generalization of this result). 

\begin{theorem}\label{SV20} (Srinivasan-Vinroot) 
    Let $(\mathbf{G},F)$ be as in the above setting and further assume $\mathbf{G}$ has connected center. Let $L$ be any subfield of $\mathbb{C}$,  let $m$ be the exponent of $\mathbf{G}^F$, and let $K = \Q(\xi_m)\cap L$. For any $\sigma \in \Gal(\Q(\xi_m)/K)$ let $r_\sigma \in \mathbb{Z}^+$ such that $(r_\sigma,m) = 1$ and $\sigma(\xi_m) = \xi_m^{r_\sigma}.$

    We can then choose  Jordan decomposition maps such that for any $\chi  \in \mathbf{G}^F$ with $\chi \in \mathcal{E}(\mathbf{G}^F,s_0)$, if $J_{s_0}(\chi) = \nu$, we have that $\Q(\chi) \subseteq L$ if and only if the following hold for every $\sigma \in \Gal(\Q(\xi_m)/K)$:
    \begin{enumerate}
        \item The element $s_0$ is $(\mathbf{G}^*)^{F^*}$-conjugate to $s_0^{r_\sigma}.$ 
        \item If $h_0 \in (\mathbf{G^*})^{F^*}$ such that $h_0s_0h_0^{-1} = s_0^{r_\sigma}$, then $^{h_0}\nu =  \,^\sigma\nu.$
    \end{enumerate}
\end{theorem}
\begin{proof}
    This is \cite[Corollary 5.1]{SV20}.
\end{proof} 

Throughout the remainder of this article, we take Jordan decomposition maps to be as in Theorem \ref{SV20} when discussing cases where $\mathbf{G}$ has connected center. If $\mathbf{G}$ is such a group and $\chi \in \mathcal{E}(\mathbf{G}^F,s)$ such that $J_s(\chi) = 1,$ we call $\chi$ a \textit{semisimple character} (note that, in the case where $\mathbf{Z}(\mathbf{G})$ is not connected, this definition of semisimple character is not equivalent to the general definition. See \cite[Section 2.6]{GM20}).

\subsection{Central Characters}

As we are concerned with simple groups of the form $G/\mathbf{Z}(G)$ for finite groups of Lie type $G$, we now collect results on the values of irreducible characters of groups of Lie type on their centers.
\begin{proposition}\label{omega}
    Let $\mathbf{G}$ be a connected reductive algebraic group and $F:\mathbf{G}\rightarrow \mathbf{G}$ a Frobenius map. Further, let $(\mathbf{G}^*,F^*)$ be dual to the pair $(\mathbf{G},F)$.
         Then for any  semisimple element  $s \in 
        (\mathbf{G}^*)^F $ there exist a character $\omega(s) \in \Irr(\mathbf{Z}(\mathbf{G})^F)$, depending only on the conjugacy class of $s$, such that for all $\chi \in  \mathcal{E}(\mathbf{G}^F,s)$, we have $\operatorname{Res}^{\mathbf{G}^F}_{\mathbf{Z}(\mathbf{G}^F)} \chi = \chi(1)\omega(s)$. In particular, if the pair $(\mathbf{T},\theta)$, with $\mathbf{T}$ an $F$-stable maximal torus in $\mathbf{G}$ and $\theta \in \Irr(\mathbf{T}^F)$, are in the geometric conjugacy class corresponding to the conjugacy class of $s$, then $\omega(s) = \theta_{\mathbf{Z}(\mathbf{G}^F)}.$
        \end{proposition}
        \begin{proof}
            This is \cite[Lemma 2.2]{Mal07}.
        \end{proof}

        The next proposition in a strengthening of \cite[Lemma 4.4]{NT13} and our argument follows the proof of that result closely.

        \begin{proposition}\label{NT13}
            Let ($\mathbf{G},F)$ and $(\mathbf{G}^*,F^*)$ be as in the statement of Proposition \ref{omega}. Denote $G := \mathbf{G}^F$ and $G^* := (\mathbf{G}^*)^{F^*}$. Let $s \in \mathbf{O}^{p'}(G^*)$ be semisimple  and let  $\omega(s)$ be  as in Proposition \ref{omega}. Then $\omega(s)$ is the principal character.
        \end{proposition}
        \begin{proof}
           Let $\mathbf{T}$ be any $F$-stable maximal torus of $\mathbf{G}$ and let $z\in \mathbf{Z}(\mathbf{G}^*)^{F^{*}}$. Then, the character in $  \mathcal{E}(G,z)$ corresponding to the trivial character under a choice of  Jordan decomposition map must be a linear character. In fact, the remarks preceding \cite[Proposition 2.5.20]{GM20} imply that this character must be trivial on $\mathbf{O}^{p'}(G^*)$. Thus, we can argue as in the first paragraph of the proof of \cite[Lemma 4.4]{NT13} to obtain that, if the pair $(\mathbf{T},\theta_z)$ is in the geometric conjugacy class corresponding to $z$, then $\theta_z$ is trivial on $\mathbf{T^*}^{F^*} \cap \mathbf{O}^{p'}(G^*)$. Since \cite[Proposition 11.4.12]{DM20} implies that $|G^*/\mathbf{O}^{p'}(G^*)| = |\mathbf{Z}(G)|$, we can argue as in the proof of \cite[Lemma 4.4]{NT13} to see that as $z$ varies over $\mathbf{Z}(\mathbf{G}^*)^{F^{*}}$, $\theta_z$ varies over $\Irr(\mathbf{T}^F/\mathbf{O}^{p'}(G)\cap \mathbf{T}^F)$. The remainder of the proof follows exactly as in the proof of \cite[Lemma 4.4]{NT13}.
        \end{proof}

\begin{lemma}\label{dualgeometry}
    Let $\mathbf{G}$ be a connected reductive algebraic group with connected center and $F:\mathbf{G}\rightarrow\mathbf{G}$ a Frobenius endomorphism. Let $(\mathbf{G}^*,F^*)$ be a dual pair to $(\mathbf{G},F)$. Further let $\mathbf{T}$ be an $F$-stable maximal torus and $\mathbf{T}^*$ its dual torus.  Then the following Diagram commutes:

    \[\begin{tikzcd}
(\mathbf{T}^*)^{(F^{*d})} \arrow{r}{}  & \Irr(\mathbf{Z}(\mathbf{G})^{F^d})  \\
(\mathbf{T}^*)^{F^{*}} \arrow{r}{} \arrow[u,hook,"i"]& \Irr(\mathbf{Z}(\mathbf{G})^{F}) \arrow{u}{- \circ N_{F^d/F}}
\end{tikzcd}
\]
where $i$ is the inclusion map, $N_{F^d/F}$ is the norm map given by $N_{F^d/F}(t) = tF(t)...F^{d-1}(t)$ and the the remaining two maps are the maps described in Proposition \ref{omega}. 
\end{lemma}

\begin{proof}
    Let $s \in (\mathbf{T}^*)^{F^*}$. Using \cite[Lemma 2.5.7]{GM20} we obtain  canonical isomorphisms $\psi:(\mathbf{T}^*)^{(F^*)}\rightarrow \Irr(\mathbf{T}^{F})$ and $\psi_d: (\mathbf{T}^*)^{(F^{*d})}\rightarrow \Irr(\mathbf{T}^{F^d})$. Let $\mathbf{T}_0$ be a maximally split torus in $\mathbf{G}$.  Then there exists a pair $(\lambda,n) \in \Irr(\mathbf{T}_0) \times \Z^{+}$ as in \cite[2.4.5]{GM20} such that $s$ corresponds to the pair $(\lambda,n)$  as described in  \cite[Lemma 2.5.7]{GM20}. Then, if we identify $\mathbf{T}^F$ with $\mathbf{T}_0[w_0]$ for appropriate choice of $w_0\in \mathbf{N}_\mathbf{G}(\mathbf{T})$ (see \cite[1.6.4]{GM20}), the image of $s$ under $\psi$ is $\lambda_0|_{(\mathbf{T})^{F}}$, where $\lambda_0 \in \Irr(\mathbf{T_0})$ and for all $t \in \mathbf{T}$ we have $\lambda_0(t^n) = \lambda(F(t)w_0^{-1}t^{-1}w_0)$. Similarly  the image of $s$ under $\psi_d$ is $\lambda_1|_{(\mathbf{T})^{F^{d}}}$, where $\lambda_1 \in \Irr(\mathbf{T})$ such that for all $t \in \mathbf{T}$ we have $\lambda_1(t^n) = \lambda(F^d(t)w_1^{-1}t^{-1}w_1)$ for some $w_1 \in \mathbf{N}_\mathbf{G}(\mathbf{T})$. Let $z \in \mathbf{Z}(\mathbf{G})^{F^d}$. Then, since $\mathbf{Z}(\mathbf{G})$ is connected, it is a torus. Since we have assumed the field over which we are working is algebraically closed, this implies there exists $z_0 \in \mathbf{Z}(\mathbf{G}) \le \mathbf{T}$ such that $z_0^n = z$. Then, 
    \begin{align*}
    \lambda_1(z) &= \lambda(F^d(z_0)w_1^{-1}z_0^{-1}w_1)\\ &= \lambda(F^d(z_0)z_0^{-1}) \\
    &= \lambda(F^d(z_0)(F(z_0)...F^{d-1}(z_0))(F(z_0)...F^{d-1}(z_0))^{-1}z_0^{-1})\\
   &=\lambda(F(N_{F^d/F}(z_0))N_{F^d/F}(z_0)^{-1})\\ 
   &= \lambda(F(N_{F^d/F}(z_0))w_0^{-1}N_{F^d/F}(z_0)^{-1}w_0)\\
   &= \lambda_0\circ N_{F^d/F}(z)
    \end{align*}
    with the second to last equality following from the fact that $N_{F^d/F}$ maps central elements to central elements. This implies the result.
\end{proof}

We note that the following is \cite[Proposition 2.7 (iii)]{ST23}; however, since the proof of that result is largely left to the reader, we provide a complete proof here.

\begin{proposition}\label{stthings}
    Let $\mathbf{G}$ be a connected reductive algebraic group and $F:\mathbf{G}\rightarrow \mathbf{G}$ a Frobenius endomorphism. Further let $(\mathbf{G}^*,F^*)$ be dual to the pair $(\mathbf{G},F)$. 
     There is a well defined group isomorphism $\Phi:(\mathbf{G}^*)^{F^*}/\mathbf{O}^{p'}((\mathbf{G}^*)^{F^*}) \rightarrow \Irr(\mathbf{Z}(\mathbf{G}^F))$ given by $s\mathbf{O}^{p'}((\mathbf{G}^*)^{F^*}) \mapsto \omega(s)$, where $s \in (\mathbf{G}^*)^{F^*}$ is a semisimple coset representative and $\omega(s)$ is as described in Proposition \ref{omega}.
\end{proposition}

\begin{proof}

We write $G := \mathbf{G}^F$ and $G^*:= (\mathbf{G^*})^{F^*}$.
We first consider the case where $\mathbf{G}$ has connected center.
    We begin by showing the map is well defined. Let $s$ and $t$ be semisimple coset representatives for the same coset of $\mathbf{O}^{p'}(G^*)$. Let $\mathbf{T}_s$ and $\mathbf{T}_t$ be $F^*$-stable maximal tori  such that $s \in\mathbf{T}_s$ and $t \in \mathbf{T}_t$. Then, since all maximal tori are conjugate, there exists some $g \in \mathbf{G}^*$ such that $t' := gtg^{-1} \in \mathbf{T}_s$ and there is some $d \in \mathbb{Z}^+$ such that $(F^*)^d(g) = g $. For the sake of clarity we write $\omega(s)$ to denote the image of $s$ under the map described in Proposition \ref{omega} when viewed as an element of $(\mathbf{G}^*)^{F^*}$ and $\omega_d(s)$ to denote the image of $s$ under the map described in Proposition \ref{omega} when viewed as an element of $(\mathbf{G}^*)^{F^{*d}}$. Then, since $\omega_d$ is constant on $(\mathbf{G}^*)^{F^{*d}}$-conjugacy classes,  we  have $\omega_d(t') = \omega_d(t)$ . Further, since $s$ and $t'$ lie in the same $F^*$-stable maximal torus, \cite[Lemma 2.5.7]{GM20} and Proposition \ref{omega} imply that $\omega_d(s)^{-1}\omega_d(t)=\omega_d(s^{-1})\omega_d(t') = \omega_{d}(s^{-1}t')$. Note that $s^{-1}t' = s^{-1}tt^{-1}gtg^{-1} \in \mathbf{O}^{p'}((\mathbf{G}^*)^{F^{*d}})$, because the commutator subgroup $((\mathbf{G}^*)^{F^{*d}})'$  is contained in $  \mathbf{O}^{p'}((\mathbf{G}^*)^{F^{*d}})$. Therefore, Proposition \ref{NT13} implies $\omega_{d}(s^{-1}t')$ is the trivial character. It follows that $\omega_d(s)  = \omega_d(t')=\omega_d(t)$. Then, Lemma \ref{dualgeometry} and the fact that  $N_{F^d/F}$ maps $\mathbf{Z}(\mathbf{G})^{F^d}$ onto $\mathbf{Z}(\mathbf{G})^{F}$ imply that $\omega(s) = \omega(t)$ as desired.

The map $\Phi$ is clearly surjective.  We further see that dualizing the argument in \cite[Proposition 11.4.12]{DM20} 
    gives that $ |G^*/\mathbf{O}^{p'}(G^*)| = |\mathbf{Z}(G)| = |\Irr(\mathbf{Z}(G))|$, so $\Phi$ must therefore be a bijection. Lastly to show that $\Phi$ is a homomorphism, let $s,t \in G^*$ be semisimple elements. Then by \cite[Proposition 1.5.13(c)]{GM20} we see that there exists a semisimple element $t' \in t\mathbf{O}^{p'}(G^*)$ such that $t'$ lies in an $F^*$-stable maximal torus containing  $s$. Then $\Phi(s)\Phi(t) = \omega(s)\omega(t)= \omega(s)\omega(t')$ by the preceding paragraph,  and, since $s$ and $t'$ lie in the same $F^*$-stable maximal torus, we can argue as above to get  $\omega(s)\omega(t') = \omega(st')$  . Further, since $t^{-1}s^{-1}st' = t^{-1}t'$ lies in  $\mathbf{O}^{p'}(G^*)$, we have $\omega(st) = \omega(st') =\omega(s)\omega(t') = \omega(s)\omega(t) $, which gives the claim.

 Now assume $\mathbf{G}$ does not have connected center. Let $\iota:\mathbf{G}\rightarrow \tilde{\mathbf{G}}$ be a regular embedding. Let $\tilde{G} = \tilde{\mathbf{G}}^F$,    and $\tilde{G^*}= (\mathbf{\tilde G}^*)^F$, where $\mathbf{\tilde{G}}^*$ is dual to $\mathbf{\tilde{G}}$.  Then \cite[Remark 1.7.6]{GM20} gives $Z:=\mathbf{Z}(\mathbf{G})^F = \tilde{\mathbf{G}}\cap \mathbf{Z}(\tilde{\mathbf{G}})^F$ (note that \cite[Proposition 3.6.8]{Car85} gives that $\mathbf{Z}(\mathbf{G}^F) = \mathbf{Z}(\mathbf{G})^F$).  Since $\iota$ is a central isotopy, as in \cite[Lemma 1.7.12]{GM20} we obtain a dual  isotypy  $\iota^*: \tilde{\mathbf{G}}^*\rightarrow \mathbf{G}^*$ which is surjective with surjective restriction $\iota^*: \tilde{G}^*\rightarrow G^*$.  First to show the described map is well defined let $s_1,s_2$ be semisimple elements of $G^*$ which lie in the same coset of $\mathbf{O}^{p'}(G^*)$. It must be the case that $  \mathbf{O}^{p'}(G^*)  \subseteq \iota^*(\mathbf{O}^{p'}(\tilde{G}^*))$, so,  we can find semisimple elements  $\tilde s_1,\tilde s_2 \in \tilde{G}^*$ such that $\iota^*(\tilde s_1) = s_1,\iota^*(\tilde s_2) = s_2$ and $\tilde s_1\tilde s_2^{-1} \in \mathbf{O}^{p'}(\tilde{G}^*)$.  Then, using \cite[Prop 2.6.16]{GM20} we see that for every $\tilde \chi \in \mathcal{E}(\tilde{G},\tilde s_i)$ the constituents of some $\tilde \chi|_G $ lie in $\mathcal{E}(G,s_i)$. Let $\chi$ be one such constituent of $
    \tilde\chi_G$.  Then $\chi|_Z$ is a constituent of $\tilde\chi|_Z$. This implies that $\omega_s  =\omega_{\tilde s}|_{Z} $. It follows, since $\omega_{\tilde s_1} = \omega_{\tilde s_2}$, that $\omega_{s_1} = \omega_{s_2}$ as desired.
    
    We then argue as in the connected center case to see that the mapping  is a bijection. So we need only show the homomorphism property. Let $s_1$ and $s_2$ be semisimple elements in $G^*$ and $\hat s_1, \hat s_2$ be as above. Then $\omega_{s_1s_2} = \omega_{\tilde s_1\tilde s_2}|_Z= \omega_{\tilde s_1}|_Z\omega_{\tilde s_2}|_Z = \omega_{s_1}\omega_{s_2}$.
    \end{proof}

    \begin{corollary}\label{center}
         Let $\mathbf{G}$ be a connected reductive algebraic group and $F:\mathbf{G}\rightarrow \mathbf{G}$ a Frobenius map. Further let $(\mathbf{G}^*,F^*)$ be dual to the pair $(\mathbf{G},F)$. Then a character $\chi \in \mathcal{E}(\mathbf{G}^F,s)$ is trivial on $\mathbf{Z}(\mathbf{G}^F)$ if and only if $s \in \mathbf{O}^{p'}((\mathbf{G}^*)^{F^*})$
    \end{corollary}

    \begin{proof}
        This follows immediately from Proposition \ref{stthings}.
    \end{proof}

    \begin{section}{ Linear and Unitary Groups}
    
We begin this section with some results on $\GL^\e_n(q)$ which will be useful for computing the fields $\Q(\GL^\e_n(q)).$
    Note that Theorem \ref{SV20} implies that if $\chi \in \mathcal{E}(\GL_n^\e(q),s)$ has the property that $J_s(\chi) = 1$ (recall that such characters are called \textit{semisimple characters}), then  $\sigma \in \G$ fixes $\chi$ if and only if $\sigma(s) $ is $\GL_n^\e(q)$-conjugate to $s$. As this is more convenient than the general case, we prove the following, which reduces the problem of computing the fields $\Q(\GL_n^\e(q))$ to the problem of computing the fields $\Q(\{\chi(g): \chi \in \Irr(\GL^\e_n(q)),\chi \text{ semisimple}, g \in \GL^\e_n(q)\})$. 
    
    \begin{proposition}\label{reduction}
Let $\chi \in \Irr(\GL_n^\e(q))$ and $\sigma \in \G$ such that $\sigma(\chi) \neq \chi$. Then there exists some $\chi_0 \in \Irr(\GL_n^\e(q))$ such that $\sigma(\chi_0)\neq \chi_0$ and $\chi_0 \in \mathcal{E}(\GL^\epsilon_n(q),s_0)$ for some semisimple element $s_0$ with $J_{s_0}(\chi_0) = 1$. 

    \end{proposition}
    \begin{proof}
    Let $\chi \in \Irr(\GL_n^\epsilon(q))$. Further, let $s \in \GL_n^\e(q)$ be semisimple and $\sigma \in \G$ such that $\chi \in \mathcal{E}(\GL_n^\epsilon(q),s)$, $J_s(\chi) = \nu$, and $\sigma(\chi) \neq \chi$. If $\sigma(s)$ is not conjugate to $s$, then it is immediate from Theorem \ref{SV20} that the character $\chi_0 \in \mathcal{E}(\GL_n^\e(q),s)$ such that $J_s(\chi_0) = 1$ is not $\sigma$-fixed, which gives the result in that case.

    Now assume $\sigma(s)$ is conjugate to $s$. Then, by Theorem \ref{SV20}, there exists $h \in \GL^\epsilon_n(q)$ with $hsh^{-1} = \sigma(s) $ and $^{h}\nu \neq \,^\sigma\nu$. Since all unipotent characters of $\GL_n^\epsilon(q)$ are rational valued, the second condition is equivalent to the condition $^{h}\nu \neq \nu$. The centralizer of $s$ in $\GL_n^\e(q)$
 is of the form $\prod_{i=1}^k\GL_{e_i}^{\e_i}(q^{r_i})$ for some $e_i,r_i\in \Z^+$ and $\e_i \in \{\pm 1\}$ (for details see \cite[Proposition 1A]{FS82}). If the conjugation action of $h$  stabilizes each component of this product, then $\nu$ must also be fixed by the action of $h$, since unipotent characters in $\GL_{n_i}^{\e_i}(q^{n_i})$ are invariant under automorphisms and the unipotent characters of this direct product are all tensor products of unipotent character of the individual terms. Therefore, there must  be isomorphic terms in the product which the action of $h$ permutes. Up to a change of basis we may assume $e_1=e_2,$ $r_1=r_2,$ and the conjugation action of $h$ maps the $\GL_{e_1}^{\e_1}(q^{r_1})$ component onto the $\GL_{e_2}^{\e_2}(q^{r_2})$ component of $\mathbf{C}_{\GL^\e_n(q)}(s)$. This gives that the characteristic polynomial of $s$ has distinct irreducible factors $f_1$ and $f_2$ with the same degree and the same multiplicity. The roots of $f_1$ must be of the form $\zeta,\zeta^q,...,\zeta^{q^{e_1-1}}$ and the roots of $f_2$ must be of the form $\zeta',\zeta'^q,...,\zeta'^{q^{e_1-1}}$ for some roots of unity $\zeta$ and $\zeta'$. Since the action of $h$ permutes these components, it must be the case that $\sigma$ maps $\zeta$ to one of the roots of $f_2$. Since $f_1$ and $f_2$ are distinct irreducible polynomials we have that they cannot have roots in common. If we then consider the semisimple element $s_0 \in \GL_n(q)$ with non-identity eigenvalues $\zeta',...,\zeta'^{q^{e_1-1}}$, we see that $\sigma(s_0)$ will not be conjugate to $s_0$. Thus, the semisimple character in $\mathcal{E}(\GL_n(q),s_0)$ is sufficient.\end{proof}

    \begin{remark}\label{tftgt} Let $s \in \GL^{\epsilon}_n(q)$ be semisimple.
        Then by 
        Theorem \ref{SV20}  we see that a semsimple character $\chi \in \mathcal{E}(\GL_n(q),s)$ is fixed by some $\sigma \in \G$ if and only if  $\sigma(s)$  is conjugate to $ s$. It follows that, if the eigenvalues (with multiplicity) of $s$ are $((\zeta_{a_1})^{k_1},...,(\zeta_{a_n})^{k_n})$, then we have that $\Q(\chi)$ corresponds to $\Stab_{\G}(\{\xi_{a_1}^{k_1},...,\xi_{a_n}^{k_n}\})$ under the fundamental theorem of Galois theory, where we take $\G$ to act on multisets of elements of $\Q^{ab}$ element-wise. We take $\G$ to act on such multisets throughout.
    \end{remark}

\begin{lemma}\label{SplitSemisimples}
    Let $s \in \GL^\epsilon_n(q)$ be a semisimple element such that $s$ has eigenvalues (with multiplicity) $(a_1,...a_k,b_1,...,b_m)$ and further let $s_1,s_2$ be semisimple elements which have eigenvalues $(a_1,...,a_n,1,...1)$ and $(b_1,...b_m,1,...,1)$ respectively. Then, if $\chi_0 \in \mathcal{E}(\GL^{\epsilon}_n(q),s),\chi_1 \in \mathcal{E}(\GL^{\epsilon}_n(q),s_1)$ and $\chi_2 \in\mathcal{E}(\GL^{\epsilon}_n(q),s_2)$ are all semisimple characters, we have $\Q(\chi_0)\subseteq \Q(\chi_1,\chi_2) $
\end{lemma}

\begin{proof}
     Assume $\sigma$ fixes the elements of $\Q(\chi_1,\chi_2)$. Then, as discussed in Remark \ref{tftgt} above, we see that $\sigma(s_1) $ is conjugate to $ s_1$ and $\sigma(s_2) $ is conjugate to $ s_2$.  It is immediate from that remark that $\sigma(s) $ is conjugate to  $ s$ and $\sigma$ fixes $\Q(\chi_0)$. It follows that $\operatorname{Stab}_{\G}(\Q(\chi_1,\chi_2)) \le \operatorname{Stab}_{\G}(\Q(\chi_0)) $ and the result follows from the fundamental theorem of Galois theory.
\end{proof}

\begin{lemma}\label{divides}
    Let $s_1,s_2 \in G =   \GL^{\epsilon}_n(q)$ be a semisimple elements such that the complete collection of  non-identity eigenvalues, with multiplicity, of $s_1$ is $(\zeta_{q^k-\e^k},\zeta_{q^k-\e^k}^{\e q},...,\zeta_{q^k-\e^k}^{(\e q)^{k-1}})$ for some $k \in \mathbb{Z}^+$ and  the complete collection of  non-identity eigenvalues, with multiplicity, of $s_2$ is $(\zeta_a,\zeta_a^{\e q},...,\zeta_a^{(\e q)^{k-1}})$ for some integer $a \in \Z^+$. If $\chi_1 \in \mathcal{E}(\GL^{\epsilon}_n(q),s_1)$ and $\chi_2 \in\mathcal{E}(\GL_n^\epsilon(q),s_2)$ are semisimple characters, then  $\Q(\chi_2 ) \subseteq \Q(\chi_1) $.
\end{lemma}

\begin{proof}

First we claim that $a$ must divide $(q^k-\e^k).$ Let $m$ be minimal such that $a|(q^m - \e^m) $, or equivalently, that $(\e q)^m \equiv 1 $(mod $a
)$.  It is well known that in this situation, since $\zeta_a$ is an eigenvalue of $s_2$, that $\zeta_a,...,\zeta_a^{{(\e q)}^{m-1}}$ must all have the same multiplicity as eigenvalues of $s_2$, so we further see that $m$ must divide $k$. This gives that in fact $(\e q)^k\equiv 1 $(mod $a
)$ and that $a|( q^k - \e^k)$  as claimed.

      By Theorem \ref{SV20},  if $\sigma \in \mathcal{G}$, we have that  each $\chi_i$ is $\sigma$-fixed if and only if $\sigma(s_i)$ is conjugate to $ s_i$.  We see all the elements of the sets of non-identity eigenvalues of $s_1$ and $s_2$ lie in  single orbits in $\overline{\mathbb{F}_p}^\times$ under the action of the group generated by the  map given by $x \mapsto x^{\epsilon q}$ for $x \in \overline{\mathbb{F}_p}^\times$. In fact, since in both $\zeta_a^{(\e q)^k} = \zeta_a$ and $\zeta_{q^k-\e^k}^{(\e q)^k} = \zeta_{q^k-\e^k}$ we see that the sets must be the entire orbits. Let $\sigma \in \mathcal{G}$ such that $\sigma(s_2) $ is not conjugate to $s_2$. Then, if $\sigma(\xi_{|G|}) = \xi_{|G|}^r$, we see that $\sigma(\zeta_{a}) = \zeta_{a}^r$ and $r$ is not equivalent to a power of $\epsilon q$ mod $a$, otherwise $s_2$ would be conjugate to $\sigma(s_2)$. We also have that $\sigma(\zeta_{q^k-\e^k})=\zeta_{q^k-\e^k}^r  $. Since $r $ is not equivalent to a power of $\epsilon q$ mod $a$  and $a|(q^k-\e^k),$ it cannot be equivalent to a power of $\epsilon q$ mod $q^k-\e^k$ and it follows that the conjugacy class of $s_1$ is not $\sigma$-fixed. Therefore, $\operatorname{Stab}_{\G}(\Q(\chi_1)) \le \operatorname{Stab}_{\G}(\Q(\chi_2)) $ and the inclusion follows from the fundamental theorem of Galois theory. 
\end{proof}

\begin{lemma}\label{ones}
   Let $G = \GL^{\epsilon}_n(q)$,  $\chi \in \mathcal{E}(G,s)$ be a semisimple character and $\sigma \in \G$ such that $\sigma(\chi) \neq \chi$ and $\sigma \in \operatorname{Stab}_{\G}(\xi_{q-{\epsilon}})$. Then there exists a semisimple element $s' \in G$ such that, if $\chi'$ is the  semisimple character in $ \E(G,s')$, then  $\sigma(\chi') \neq \chi'$ and none of the non-identity eigenvalues of $s'$ lie in $\GL_1^\epsilon(q)$.
\end{lemma}

\begin{proof}
    There exists $r\in \mathbb{N}$ such that $\sigma(\zeta_{|G|}) = \zeta_{|G|} ^r$. Since $\sigma \in\operatorname{Stab}_{\G}(\zeta_{q-\epsilon})$ we have $r \equiv 1 \text{ mod } (q-\epsilon)$. Let $\{a_1,..,a_k,b_1,...b_m\}$ be the multiset of eigenvalues of $s$, where each $b_i \in \GL_1^\epsilon(q)$ and none of the $a_i$ are. Then the multiset of  eigenvalues of $\sigma(s)$ is $\{a_1^r,..,a_k^r,b_1^r,...,b_m^r\} = \{a_1^r,..,a_k^r,b_1,...b_m\}$. This implies that $\{a_1,...,a_k\} \neq\{a_1^r,..,a_k^r\} $.  Let $ s' $ be a semisimple element with the multiset of eigenvalues  $\{a_1,..,a_k,1,...,1\}$. Then $\sigma(s')$ has multiset of eigenvalues  $\{a_1^r,..,a_k^r,1,...,1\} $ which is not equal to the multiset $ \{a_1,..,a_k,1,...,1\}$. This implies that $s'$ has the desired properties. \end{proof}

\begin{proposition}\label{TypeAinclusions}
    Let $n$ be and integer with $n \ge 2$ and let $q$ be a power of a prime. Then the following hold:

    \begin{enumerate}
        \item[(i)]  $\Q(\GL^{\epsilon}_{n-1}(q))\subseteq\Q(\Gl^{\epsilon}_{n}(q)) $.
        \item[(ii)] $\Q(\GL^{\epsilon}_{n}(q))\subseteq\Q(\SL^{\epsilon}_{n+1}(q)) $.
    \end{enumerate}

\end{proposition}

\begin{proof}
 Let $\tilde{G}_0 = \GL^{\epsilon}_{n-1}(q), \tilde{G} = \GL^{\epsilon}_{n}(q)$ and $G = \SL^{\epsilon}_{n+1}(q)$. Let $\sigma \in\G$ such that $\sigma$ does not fix $\Q(\tilde{G_0})$ element-wise. To show $(i)$
 it is sufficient to find a character in $\Irr(\tilde{G})$ which is not $\sigma$-fixed.  By Proposition \ref{reduction} we can find some semisimple character $\chi_0 \in \Irr(\tilde{G}_0)$ such that $\sigma(\chi_0)\neq \chi_0$. Then, Remark \ref{tftgt} implies that if $\chi_0 \in \mathcal{E}(\tilde G_0,s)$, then $\sigma(s)$ is not conjugate to $s$. We then define $\hat s = \mathrm{diag}(s,1 )\in \tilde{G}$ and  consider  $\theta \in \mathcal{E}(\tilde{G},\hat s)$.  We see that  $\sigma(\theta) \in \mathcal{E}(\tilde{G},\sigma(\hat s))$ and it is immediate that  $\sigma(\hat s) $ is not conjugate to $\hat s$ and that $\theta$ is not $\sigma$-fixed. Thus, $\operatorname{Stab}_{\G}(\Q(\tilde{G}_0)) \ge \operatorname{Stab}_{\G}(\Q(\tilde{G})) $ and the fundamental theorem of Galois theory gives the result.

 To show $(ii)$ first assume that $\sigma (\xi_{q-\epsilon})= \xi_{q-\e}$ and $\sigma$ does not fix $\Q(\tilde{G})$ element-wise. Then from Lemma \ref{ones} we can find a semisimple character $\chi \in \Irr(\tilde G)$ such that $\sigma(\chi) \neq \chi$ and  $\chi \in \mathcal{E}(\tilde{G},s)$ such that 
 none of the non-identity eigenvalues of $ s$ lie $\GL_1^{\epsilon}(q)$. Let $\hat s = \mathrm{diag}(s,1)$. Then,  as in the previous case, we see that $\hat{s}$ is not conjugate to $\sigma(\hat{s})$. 
 Further, since none of the non-identity eigenvalues of $\hat s$ lie in $\GL_1^\e$
 and $1$ is an eigenvalues of $\hat s$, we see that, for any non-trivial $z \in \mathbf{Z}(\tilde{G})$, $z\hat s$   will have a non-identity eigenvalue which lies in $\GL_1^\e$ and therefore cannot be conjugate to $\sigma(\hat s)$. Thus, if we let $\pi:\GL_{n+1}^\e(q) \rightarrow \PGL_{n+1}^\e(q)$ be the canonical projection map (recall that $\PGL_n^\e(q)$ is dual to $\SL^\e_n(q)$), we see that any character in the series $\mathcal{E}(G,\pi(\hat{s}))$ cannot be fixed by $\sigma$ since $\sigma(\pi(\hat s)) $ is not conjugate to $\pi(\hat s)$.

 Now consider any $\sigma \in \G$ such that  $\sigma (\xi_{q-\epsilon})\neq\xi_{q-\e}$. Consider the semisimple element $s \in \GL_{n+1}^\e(q)$ whose only non-identity eigenvalue is $\zeta_{q-\e}$ with multiplicity $1$. Then for any character $\chi \in \mathcal{E}( \GL_{n+1}^\e(q),s)$ we have $\sigma(\chi) \neq \chi$. We see that the only eigenvalues of $\sigma( s)$ are $\zeta_{q-\epsilon}^r\neq \zeta_{q-\epsilon}$ and $1$,  for some $r \in (\Z/(q-\e)\Z)^\times$. Further, since $n +1 \ge 3$ we see that the eigenvalue of $1$ has multiplicity at least $2$. It follows that $z\hat s$ and $\sigma(\hat s)$ cannot be conjugate for any $z \in \mathbf{Z}(\tilde{G})$. So, if we take $\pi$ as in the previous case we see that any character in $\mathcal{E}(G,\pi(s))$ cannot be fixed by $\sigma$. Thus, we see similarly as above that $\operatorname{Stab}_{\G}(\Q(\tilde{G}_0)) \ge \operatorname{Stab}_{\G}(\Q(G)) $ and the fundamental theorem of Galois theory gives the result.
\end{proof}

\subsection{Some Number-Theoretic Lemmas}

\begin{lemma}\label{NumberTheoryLemma}
       Let $n \ge 3$ be an integer. Then $|\Q(\xi_n):\Q(\xi_n+\xi_n^{-1})| = 2$. Further, if $q$ is a power of a prime and $\e \in \{\pm1\}$, we have  $|\Q(\xi_{q^2-1}):\Q(\xi_{q^2-1} + \xi_{q^2-1}^{\e q})| = 2$ unless $q = 3$ and $\e = -1$
    \end{lemma}

    \begin{proof}
        For the first equality it is sufficient to show that the only  non-trivial Galois automorphism in $\Gal(\Q(\xi_n)/\Q)$ which fixes $\xi_{n} + \xi_{n}^{-1} $ and not $\xi_{n}$ is given by $\xi_{n} \mapsto \xi_{n}^{-1}$. Assume $\sigma$  fixes $\xi_{n} + \xi_{n}^{-1} $. Then $\xi_{n} + \xi_{n}^{-1}   -\sigma(\xi_n) - \sigma(\xi_{n}^{-1}) = 0$. It is know that there a no minimal vanishing sums of roots of unity with exactly $4$ terms and  the only minimal vanishing sums of roots of unity of size $2$ are of the form  $\xi-\xi$ for $\xi$ a root of unity. Then, either $\sigma(\xi_{n}) \in \{\xi_{n},\xi_{n}^{-1}\} $ or $\xi_{n} = - \xi_{n}^{-1}$.  The second possibility can only happen when $n$ is even and $n/2+1 = n-1$. Therefore, this situation only happens when $n = 4$, in which case $\xi_{n} + \xi_{n}^{-1}= 0 $ and  $|\Q(\xi_{n}):\Q(\xi_{n} + \xi_{n}^{-1} )| = |\Q(\xi_{n}):\Q|= 2$, so the claim still holds. 

         For the second claim, it is sufficient to show that the only non-trivial Galois automorphism in $\Gal(\Q(\xi_{q^2-1})/\Q)$ which fixes $\xi_{q^2-1} + \xi_{q^2-1}^{\e q}$ is the Galois automorphism given by $\xi_{q^2-1} \mapsto \xi_{q^2-1}^{\e q}$. Let $\sigma$ be a Galois automorphism which fixes $\xi_{q^2-1} + \xi_{q^2-1}^{\e q}$. Then, arguing exactly as in the first paragraph of this proof, either $\sigma(\xi_{q^2-1}) \in \{\xi_{q^2-1},\xi_{q^2-1}^{\e q}\}$ or $\xi_{q^2-1} = - \xi_{q^2-1}^{\e q}$. If we are in the second case and we have that $\e = 1$, we get that $q = \frac{q^2-1}{2} + 1$, which has no solutions where $q$ is a power of a prime. If instead $\e = -1$ we get that  $q^2 -1 - q = \frac{q^2-1}{2} + 1$. The only solution of this equation where $q$ is a power of a prime is when $q =3$, which we assumed was not the case. 
    \end{proof}

    \begin{lemma}\label{numbertheorylemmabig}
    
        Let $n\ge 2$ be an integer, let $\e \in \{\pm1\}$ and let  $q$ be a power of a prime such that $(n,q,\e) \notin \{(3,2,-1),(2,3,-1)\}$. Let $r \in (\mathbb{Z}/(q^n-\e^n)\Z)^\times$ and $\sigma \in \G$ such that $\sigma(\xi_{q^n-\e^n}) = \xi_{q^n-\e^n}^r$. Then $\sigma$ fixes $\xi_{q^n-\e^n}+ \xi_{q^n-\e^n}^{\e q }+ ... +\xi_{q^n-\e^n}^{(\e q)^{n-1}}$ if and only if $r \in \langle \e q\rangle$ when viewed as an element of $  (\mathbb{Z}/(q^n-\e^n)\Z)^\times$.
    \end{lemma}

\begin{proof}
    We write $a:= q^n - (\e)^n$ throughout. First assume that $n\neq 2$ and $(n,q,\e) \neq (6,2,1)$. Let $A := \xi_{a}+...+\xi_{a}^{(\e q)^{n-1}}$. Throughout, if $r \in (\Z/a\Z)^\times$, we let $\sigma_r$ denote the element of  $\Gal(\Q(\xi_{a})/\Q)$ such that $\sigma(\xi_{a}) = \xi_a^r$.   It is clear that if $r \equiv (\e q)^k \text{( mod } a)$ for some $k$, then $\sigma_r$ fixes $A$.  It is then sufficient to show that if $r$ is not equivalent to a power of $\e q$ mod $a$, then $\sigma_r(A) \neq A$. Assume not, then  there exists some $r \in (\mathbb{Z}/a\Z)^\times$ such that $r$ is not equivalent to a power of $\e q$ mod $a$ and 
 $$
 A- \sigma_r(A) = \xi_{a}+...+\xi_{a}^{(\e q)^{n-1}} - \xi_{a}^r-...-\xi_{a}^{r(\e q)^{n-1}}  = 0 
 $$
 
 Then we can decompose this sum into a sum of minimal vanishing sums of roots of unity (i.e. non-trivial subsums that are equal to zero and that themselves have no non-trivial subsums equal to zero). The only vanishing sums of root of unity with 2 terms are obviously of the form $\xi + (- \xi)$ for some root of unity $\xi$, and since we assumed $r$ is not equivalent to a power of $\e q$ mod $a$ we see that these subsums only occur if $A = -A$ and $A = 0$. We proceed  with the case that $A \neq -A$, and the case with $A = -A$ will follow from a similar argument. Since there are no minimal vanishing subsums with exactly 2 terms, it follows that there must exist a minimal vanishing subsum containing $\xi_{a}^{(\e q)^k}$ and $\xi_{a}^{(\e q)^{m+k}}$, for some non-negative integer $k$ and positive integer $m< n$. Write this subsum 
 $$
 \xi_{a}^{(\e q)^k} + \xi_{a}^{(\e q)^{m+k}} + \e_1\xi_{a}^{r^{\delta_1}(\e q)^{k_1}}+...+\e_j\xi_{a}^{r^{\delta_j}(\e q)^{k_j}}
 $$
 where $\e_i \in \{\pm 1\}$,  $\delta_i \in \{0,1\}$,  and the $k_i$ are positive integers. It is an immediate corollary of  \cite[Corollary 3.2]{LL00} that the terms in 
  \[
\xi_{a}^{-(\e q)^{k}}  (\xi_{a}^{(\e q)^k} + \xi_{a}^{(\e q)^{m+k}} + \e_1\xi_{a}^{(\e q)^{k_1}}+...+\e_j\xi_{a}^{(\e q)^{k_j}})
\]
\[
= 1 + \xi_{a}^{((\e q)^k((\e q)^{m}-1))} + \e_1\xi_{a}^{(\e q)^{k_1}}\xi_{a}^{-(\e q)^{k}}+...+\e_j\xi_{a}^{(\e q)^{k_j}}\xi_{a}^{-(\e q)^{k}}
\]
 all have square free order. If we consider the term  $\xi_{a}^{((\e q)^k(\e q)^{m}-1)} = (\xi_{a}^{(\e q)^k})^{((\e q)^{m}-1)}$  we see that this implies that, for any prime $\ell$ with $\ell^2|a$, we have that $\ell$ divides $((\e q)^{m}-1)$. Due to a classical result of Zsigmondy in \cite{Zs}, we see that, since we assumed $n \neq 2$ and $(n,q,\e) \notin\{(6,2,1),(3,2,-1)\}$, there exists a prime $\ell$ such that $\ell | ((\e q)^n-1)$ and $\ell \nmid ((\e q)^k-1)$ for $0< k < n$. Thus, the above implies that $\ell  || ((\e q)^n-1)$. Then we rewrite

\[
A- \sigma(A) =  \xi_{a}+...+\xi_{a}^{(\e q)^{n-1}} - \xi_{a}^r-...-\xi_{a}^{r(\e q)^{n-1}} \]
 \[
=\xi_{a_{\ell'}}\xi_{\ell}+...+\xi_{a_{l'}}^{(\e q)^{n-1}} \xi_{\ell}^{(\e q)^{n-1}}- \xi_{a_{\ell'}}^r\xi_{\ell}^r-...-\xi_{a_\ell'}^{r(\e q)^{n-1}}\xi_{\ell}^{r(\e q)^{n-1}}. 
 \]
 Since $n$ is minimal such that  $\ell|(q^n-(\e)^n)$,  we see that $\xi_{\ell},...,\xi_{\ell}^{(\e q)^{n-1}}$ are all distinct. If $\xi_{\ell}^r$ is not equal to any of $\xi_{\ell},...,\xi_{\ell}^{(\e q)^{n-1}}$, then by linear independence of primitive $\ell$th roots of unity over $\Q(\xi_{a_{\ell'}})$, we see that this sum cannot be zero and obtain a contradiction. If $\xi_{\ell}^r$ is  equal to one of  $\xi_{{\ell}},...,\xi_{\ell}^{(\e q)^{n-1}}$, then we can regroup the sum so that for some integers $l_1,..,l_{n-1}$ we have  
 $$A - \sigma(A) = \sum_{k=0}^{n-1}(\xi_{a_{\ell'}}^{(\e q)^k}-\xi_{a_{\ell'}}^{r (\e q)^{l_k}})\xi_{\ell}^{(\e q)^k}.$$ 
  Since, as discussed prior, we cannot have terms which are negatives of each other, we get  $(\xi_{a_{\ell}'}^{(\e q)^k}-\xi_{a_{\ell'}}^{r (\e q)^{l_k}}) \neq 0 $ for all $k$ and again linear independence gives a contradiction.

 Now we consider the case where $n = 2$. Since we assumed $(q,\e) \neq (3,-1)$ we see from the proof of Lemma that \ref{NumberTheoryLemma} the only non-trivial  Galois automorphism which fixes $\xi_a + \xi_{a}^{\e q}$ is the automorphism such that $\xi_a \mapsto \xi_a^{\e q}$, and arguing as above gives the result.  

 We lastly consider the case where $(n,q,\e) = (6,2,1).$ In this case $q^n-\e =63 $. Assume by way of contradiction there there exists some $r \in (\Z/63\Z)^\times \setminus\langle2\rangle $ such that
 \[
\xi_{63} + \xi_{63}^2+\xi_{63}^4+\xi_{63}^8+\xi_{63}^{16}+\xi_{63}^{32} - (\xi_{63}^r + \xi_{63}^{2r}+\xi_{63}^{4r}+\xi_{63}^{8r}+\xi_{63}^{16r}+\xi_{63}^{32r}) = 0
 \]
 If $\xi_7^r \in \{\xi_7^3,\xi_7^5,\xi_7^6\}$
we can regroup this sum as 
\[
(\xi_9-\xi_9^{r2^{k_1}})\xi_7 +(\xi_9^2-\xi_9^{r2^{k_2}})\xi_7^2 + ...+ (\xi_9^{32}-\xi_9^{r2^{k_6}})\xi_7^{32}.
\]
It is straightforward to see that, since $r  \in (\Z/63\Z)^\times \setminus\langle2\rangle $, we have $\xi_9^m - \xi_{9}^{r2^{k_m}} \neq 0$, so linear independence if $\xi_7,...,\xi_7^6$ over $\Q(\xi_9)$ gives a contradiction. In the case where  $\xi_7^r \in \{\xi_7,\xi_7^2,\xi_7^4\}$ we argue similarly by grouping the sum in to collections of $4$ terms and using fact that there are no minimal vanishing sums of root of unity with $4$  terms.\end{proof}

 \subsection{Proofs of the main Results.}
 Recall from the introduction that, if we let $d\in \mathbb{Z}^+$, $\epsilon \in \{\pm 1\}$ and we have some fixed prime power $q$, then we write

 $$
 \gamma_{\e,q}(d) := \sum_{k=0}^{d-1} \xi_{q^d-\e^d}^{(\epsilon q)^k}
 $$

We now prove Theorem \ref{main1}, the statement of which we recall here.

\main*

\begin{proof}
     Let $F:= \Q(\{\gamma_{\epsilon,q}(d): 1 \le d \le n\})$. We proceed by induction on $n$.   Note that the case of $n =1$ is trivial, so we  let $G = \operatorname{GL}^\e_n(q)$ for some $n$ and $q$ with $n >1$. Using Proposition \ref{reduction}, Lemmas \ref{SplitSemisimples} and \ref{divides} and the structures of semisimple elements in $G^* = G$, we see that it is sufficient to consider characters in series corresponding to semisimple elements with non-identity eigenvalues of the form 
    $(\zeta,\zeta^{\e q},...,\zeta^{(\e q)^{k-1}})$, where $|\zeta| = ({q}^k-\e^k)$ for $1\le k \le n$. Further, it is immediate from the characterization of the fields of values given in Remark \ref{tftgt} that, for a given $k$, it is sufficient to consider the cases where $\zeta= \zeta_{q^k-\e^k}$, as all other choices of $\zeta$ will give the same field. For $1 \le k\le n $ we let $s_k$ be the semisimple element with non-identity eigenvalues $\zeta_{q^k-\e^k},\zeta_{q^k-\e^k}^{\e q},...,\zeta_{q^k-\e^k}^{(\e q)^{k-1}}$ each with multiplicity $1$ and we let $\chi_k$ be the semisimple character in $\mathcal{E}(G,s_k).$

    First consider $k < n$. Then we can find some $s_{k}'$ in $\GL^\e_{n-1}(q)$ with the exact same list of non-identity eigenvalues with the same multiplicity. We then see that Remark \ref{tftgt} gives the exact same characterization of the field of values of the semisimple character in $\mathcal{E}(\GL^\e_{n-1}(q),s_k')$ as it does the field $\Q(\chi_k)$. Thus,  $\Q(\chi_{k}) \subseteq \Q(\GL_{n-1}^\e(q))$ and Proposition \ref{TypeAinclusions} along with the inductive hypothesis imply that we need not consider these characters.  It is then sufficient to consider the character $\chi_{n}$. Then Lemma \ref{numbertheorylemmabig} and Lemma \ref{NumberTheoryLemma} along with its proof imply that $\gamma_{\e,q}(n)$ is fixed by some Galois automorphism $\sigma$ if and only if $\sigma \in \Stab_\G(\{\xi_{q^n-\e^n},...,\xi_{q^n-\e}^{n-1}\})$ unless $(n,q,\e) \in \{(2,3,-1),(3,2,-1)\}$. Thus, if we are not in one of these cases it is then immediate from Remark \ref{tftgt} that $\Q(\chi_{n}) = \Q(\gamma_{\e,q}(n))$, which gives the result in these cases. If $(n,q,\e) \in \{(2,3,-1),(3,2,-1)\}$, then we have that  $\Stab_\G(\{\xi_8,\xi_8^5\})$ and $\Stab_\G(\{\xi_9,\xi_9^7,\xi_9^4\}) $ correspond to the fields $\Q(i)$ and $\Q(\xi_3)$ respectively. Further, in both cases we have $\gamma_{\e,q}(n) = 0$, so is straightforward to see that the result still holds in these cases.  
    \end{proof}

    \begin{corollary}\label{brauer}
        Let $q$ be a power of a prime, $\e\in\{\pm1\}$ and $n \in \Z^+$. If $\phi:\GL^\e_n(q)\rightarrow \GL_n(\overline{\mathbb{F}_q})$ is the natural embedding and $\chi_\phi$ is the corresponding Brauer character, then $\Q(\chi_\phi) = \Q(\GL^{\e}_n(q))$.
    \end{corollary}

    \begin{proof}
        Let $g \in \GL^{\e}_n(q)$. Then the set of eigenvalues of $g$, with multiplicity, will take the form
        $$
        \zeta_{a_1},\zeta_{a_1}^{\e q}....,\zeta_{a_1}^{(\e q)^{n_1 -1}},\zeta_{a_2},\zeta_{a_2}^{\e q},...,\zeta_{a_2}^{(\e q)^{n_2-1}},...,\zeta_{a_k},...,\zeta_{a_k}^{{(\e q)}^{n_k-1}}
        $$
        for integers $k,a_1,...a_k,n_1,...n_{k}$ such that each $n_i$ is minimal with respect the to the property that $\zeta_{a_i}^{(\e q)^{n_i}} = 1 $ and $\sum n_i =n$. From this description of the eigenvalues of an arbitrary element, it follows that the Brauer character $\chi_\phi$ will have field of values $$\Q(\xi_{a}+...+\xi_{a}^{{ (\e q)}^{n_0-1}}: a,n\in \mathbb{Z}^{+}, n_0 \le n, n \text{ minimal such that } a|(q^{n_0}-\e^{n_0}) )$$ 
        
        It is immediate from this and Theorem \ref{main1} that $\Q(\chi_\phi) \supseteq\Q(\GL^\e_n(q))$.
        Let  $a,n\in \mathbb{Z}^{+}$ such that  $ n_0 \le n$ and $ n $ is minimal such that $ a|(q^{n_0}-\e^{n_0}) $. Then, from Theorem \ref{main1} it is sufficient to show that $\xi_{a}+...+\xi_{a}^{{ (\e q)}^{n_0-1}} \in \Q(\xi_{q^{n_0}-\e^{n_0}}+...+ \xi_{q^{n_0}-\e^{n_0}}^{{(\e q)}^{n_0-1}})$. First assume $(n_0,q,\e) \notin \{(3,2,-1),(2,3,-1)\}$. Then by the fundamental theorem of Galois theory it is sufficient to show that any automorphism in $\G$ that fixes  $\xi_{q^{n_0}-\e^{n_0}}+...+ \xi_{q^{n_0}-\e^{n_0}}^{{\e q}^{n_0-1}}$ must also fix $\xi_{a}+...+\xi_{a}^{{ (\e q)}^{n_0-1}}$. This is immediate from \ref{numbertheorylemmabig}. If $(n_0,q,\e) \in \{(3,2,-1),(2,3,-1)\},$ then we can compute directly, using Theorem \ref{main1}, that for all choices of $a$ we have $\xi_{a}+...+\xi_{a}^{{ (\e q)}^{n_0-1}} \in\Q( \GL^\e_2(q)) \subseteq \Q(\GL^\e_n(q)).$
    \end{proof}

\begin{lemma}\label{neweigenvalues}
    Let $G = \SL^{\epsilon}_n(q)$ and let $s$ be a semisimple element of $\GL^{\epsilon}_n(q)$ such that the complete collection of non-identity eigenvalues of $s$ is $(\zeta_{a},\zeta_{a}^{\epsilon q},...,\zeta_{a}^{(\epsilon q)^{n-1}})$ with $a|(q^n-\epsilon^n)$.  Let $\chi \in \mathcal{E}(\GL_n(q),s)$ be semisimple. Then, if $b = (a,q-\epsilon)$, we have  that $\sigma \in \G$ fixes the restriction $\chi_G$ if and only if $\sigma$ fixes the multiset $X:= \{ \xi_{a/b},\xi_{a/b}^{\epsilon q} ,...,\xi_{a/b}^{(\epsilon q)^{n-1}}\}$.
\end{lemma}

\begin{proof}
    Let $s$ be as above.  Let $\sigma \in \mathcal{G}$ such that $\sigma$ stabilizes $X$. Let  $r \in \Z $ such that $\sigma(\xi_{|G|}) = \xi_{|G|}^r$. Since $\sigma$ stabilizes $X$, we have $\sigma(\xi_{a/b})=\xi_{a/b}^r = \xi_{a/b}^{(\epsilon q)^k}$ for some $k$ and   $r \equiv (\epsilon q^k)(\text{mod } a/b)$. Thus, $r = (\epsilon q^k)+ (a/b)r_0$ for some $r_0 \in \Z$. Then there exists $0 \le l < b $ such that $r_0 + l$ is divisible by $b$.  Then $z := \operatorname{diag}((\zeta_a)^{(a/b)l},...,(\zeta_a)^{(a/b)l})$ lies in the center of $\tilde{G}$. Consider $\sigma(s)z$. We see that $\sigma(s)z$ has the eigenvalue $\zeta^{((\epsilon q)^k+ (a/b)r_0)}_{a}(\zeta_a)^{(a/b)l} = \zeta_{a}^{((\epsilon q)^k + (a/b)(r_0+l))}= \zeta_{a}^{(\epsilon q)^k}$. This implies that $z\sigma(s)$ is conjugate to $ s$ and that for the semisimple character $\chi \in \mathcal{E}(\tilde{G},s)$ we have $\sigma(\chi_{G}) =\chi_G$.

    Conversely assume $\sigma \in \mathcal{G}$ fixes $\chi_G$. Then  we have that there exists $z \in \mathbf{Z}(\tilde{G})$ such that
  $\sigma(s) $ is conjugate to $ zs$.  We have that $z$ will be a scalar matrix for the scalar $\zeta_{q-\epsilon}^\ell$ for some $\ell$. Then $\sigma(\zeta_a) = \zeta_a^{(\epsilon q)^k} \zeta_a^{m_0}$ such that $\zeta^{m_0} = \zeta_{q-\epsilon}^{-\ell}$. Since the order of $\zeta_a^{m_0}$ divides both $(q-\epsilon)$ and  $a$,  we can write $m_0 = (a/b)m$ for some $m$.
  We again have that there exists some $r$ such that $\sigma(\xi_{|G|}) = \xi^r $, and the above gives that $r \equiv (\epsilon q)^k + (a/b)m (\text{mod } a)$. It follows that $\sigma(\xi_{a/b}) = \xi_{a/b}^{(\epsilon q)^k}$ and $\sigma$ fixes $X$.
\end{proof}

We now compute the fields of values of the groups $\SL_n^\e(q)$ (note that the case of $n = 2$ was considered in \cite[Theorem 1.3]{DOW23} so we don't consider it here).

\begin{proposition}\label{SL}
    Let $n \ge 3$, $q = p^a$ for some prime $p$,  $\epsilon \in \pm 1$ and, if $p\neq 2$ let $\eta \in \{\pm1\}$ such that $p \equiv \eta \mod 4$. If $a = {\frac{q^n-\e^n}{q-\e}}$ and  $X_a := \xi_a+\xi_a^{\epsilon q}+...+\xi_a^{(\epsilon q)^{n-1}}$, then

    \[\Q(\SL^{\epsilon}_n(q))= \begin{cases} 
      \Q(\GL^{\epsilon}_{n-1}(q))(X_a,\sqrt{\eta q}) & q \text{ odd, } q \text{ not a square, } 2\le n_2\le (q-\epsilon)_2\\
      \Q(\GL^{\epsilon}_{n-1}(q))(X_a) & \text{otherwise} \\
   \end{cases}.
    \]
    
\end{proposition}
\begin{proof}
 Let $G := \SL_n^\e(q)$ and $\tilde{G} := \GL^\e_n(q)$.
  Given some  irreducible character $\chi \in \Irr(\tilde{G})$ and some constituent $\theta $ of $\chi_{G}$, the work of \cite[Section 4]{Tur01} and \cite{SV19} describe the field of values of $\theta$ in terms of the field of values of $\chi_{G}$. Assume $q$ odd $q$ not a square and $2 \le n_2 \le (q -\e)_2$. Then, \cite[Theorem 6.1]{SV19} gives that if there exist  a semisimple element $s 
\in \GL^{\epsilon}_n(q)$ such that $\operatorname{diag}(\zeta_{n_2})s$  is conjugate to $s$ and $s \in G$, then $\sqrt{\eta q} \in \Q(\SL^\e_n(q))$. If $n_2 = 2^k$, then $2^{k+1}$ divides $q^2-1$.  Then, if  $s$ is the semisimple elements whose eigenvalues are all of the primitive $2^{k+1}$th roots of unity, all with equal multiplicity, then $s$ meets the requirement in  \cite[Theorem 6.1]{SV19} and $\sqrt{\eta q} \in \Q(G)$. 

We now show $\Q(\GL^\e_{n-1}(q))(X_a) \subseteq  \Q(\SL^\e_n(q))$
   Proposition \ref{TypeAinclusions} gives that $\Q(\GL^\e_{n-1}(q)) \subseteq \Q(G)$, so it is sufficient to show $X_a \in \Q(G) $. Let $s\in \GL_n^\e(q)$ such that the set of nontrivial eigenvalues of $s$ is $(\zeta,\zeta^{\epsilon q},...,\zeta^{(\epsilon q)^{n-1}})$ with $q^n-\e^n =:|\zeta|$. Then, it is immediate from Lemma \ref{neweigenvalues} that $X_a$ is contained in the field of values of the semisimple character in the series corresponding to $s$. Thus, we have $\Q(\GL^\e_{n-1}(q))(X_a,\sqrt{\eta q}) \subseteq \Q(G)$ in the case where $q \text{ odd, } q \text{ not a square, } 2\le n_2\le (q-\epsilon)_2$ and we have $\Q(\GL^\e_{n-1}(q))(X_a) \subseteq \Q(G)$ in all other situations.

   We now show the other containment. Again using \cite[Theorem 6.1]{SV19} it is sufficient to show that $\Q(\chi_G:\chi\in \Irr(\tilde{G})) \subseteq \Q(\GL^\e_{n-1}(q))(X_a)$ in all cases.   First we  consider characters in Lusztig series corresponding to semisimple elements in $\tilde{G}$ with complete collection of non-identity eigenvalues of the form $(\zeta,\zeta^{\epsilon q},...,\zeta^{(\epsilon q)^{n-1}})$ such that $|\zeta|=c$, where  $c|(q^n-\epsilon^n)$  and $c\nmid (q^k-\e^k)$ for all $1\le k <n$.  Let $s$ be such a semisimple element. Then  by \cite[Proposition 1A]{FS82}  $\mathbf{C}_{\tilde{G}}(s)$ is cyclic and has only one unipotent character, so the series $\mathcal{E}(\tilde{G},s)$ contains exactly one character.  By arguing similarly to the proof of Lemma \ref{neweigenvalues} we get that if $s \in \GL^{\epsilon}_n(q)$ has eigenvalues $(\zeta,\zeta^{\epsilon q},...,\zeta^{(\epsilon q)^{n-1}})$ with $c :=|\zeta|$ such that  $c|(q^n-\epsilon^n)$ and $\chi \in\mathcal{E}(\GL^{\epsilon}_n(q),s)$, is semisimple, then, under the fundamental theorem of Galois theory, $\Q(\chi|_G) $ corresponds to $\Stab_\G(K_{c/d})$, where $K_{c/d}: = \{\xi_{c/d},\xi_{c/d}^{\epsilon q} ,...,\xi_{c/d}^{(\epsilon q)^{n-1}}\}$ and $d =  (c,q-\epsilon)$. It is straightforward to see that if $c'|c$, $d' = (c',q-\e)$ and $c|(q^n-\e^n)$, then $\Stab_\G(K_{c/d}) \le \Stab_\G(K_{c'/d'})$. It follows that it is sufficient to consider the case where $ c  = q^n-(\epsilon)^n$ and $c/d = a$. 

  Let $s$ be as above with $c =  q^n-(\epsilon)^n$. To show $\Q(\chi_G) \subseteq \Q(X_a)$ it is then sufficient to show that if $\sigma \in \Stab_\G(X_a)$ then $\sigma \in \Stab_\G(\chi_G) =\Stab_\G(K_a)$.  We note that if a  so-called Zsigmondy prime divides $q^n-\e^n$, then it must divide $a/b = (q^n-\e^n)/(q -\e)$.  So, in cases where the Zsigmondy prime exists, we can argue exactly as in the proof of Lemma \ref{numbertheorylemmabig} to get that a Galois automorphism fixes $X_a$ if and only if it fixes $K_{a/b}$, so we need only consider the exceptions.  We can use a similar argument and  Lemma \ref{NumberTheoryLemma} for the cases with $n = 2$. The only remaining cases are then when $(n,q,\e) \in \{(6,2,1), (3,2,-1)\}$. These follow from straightforward computations similar to those which appear in the proof of Lemma \ref{numbertheorylemmabig}.

  Now we consider the case where $\chi$ is contained in a series indexed by a semisimple element which is not of the form discussed above. Then from the proof of Proposition \ref{reduction} we see that the fields of values of $\chi$ is contained in the fields of values of some semisimple character $\chi'$ in a series index by a semisimple element which is not of the form discussed above. Then, using Lemma \ref{SplitSemisimples} and Remark \ref{tftgt} we see that the fields of values of $\chi'$ is contained in a field which is a subfield of $\Q(\GL_{n-1}(q))$. Thus, if $\chi_0$ is a constituent of $\chi_G,$ then, using \cite[Theorem 6.1]{SV19}  $\Q(\chi_0) \subseteq \Q(\chi_G,\sqrt{\eta q}) \subseteq \Q(\GL_{n-1}(q))(\sqrt{\eta q})$  when $q \text{ odd, } q \text{ not a square, } 2\le n_2\le (q-\epsilon)_2$ and $\Q(\chi_0) \subseteq \Q(\chi_G) \subseteq \Q(\GL_{n-1}(q))$ otherwise.
  This concludes the proof of this containment 
\end{proof}

We now consider the simple groups $\PSL_n^{\epsilon}(q)$.


\begin{theorem}\label{PSL}
Let $n \ge 3,$ $q $ a power of a prime and $\e \in \{\pm 1\}$ such that 
$$
(n,q,\e) \notin\{(3,2,-1),(3,7,1),(3,5,-1),(3,4,1),(4,3,1),(4,5,1),(4,3,-1),(5,4,-1)\}.
$$  Further, let $a = (q^n-\e)/(q-\e)$, $b = (q-\e,a)$ and $K  = \xi_{a/b}+\xi_{a/b}^{\e q}+...+\xi_{a/b}^{(\e q)^{n-1}}$. 
   
   \[\Q(\PSL^{\epsilon}_n(q))= \begin{cases} 
      \Q(\GL^{\epsilon}_{n-1}(q))(K,\sqrt{\eta q}) & q \text{ odd, } q \text{ not a square, } 2\le n_2\le (q-\epsilon)_2\\
      \Q(\GL^{\epsilon}_{n-1}(q))(K) & \text{otherwise}. \\
   \end{cases}
\]
Further,  if $\eta'\equiv q \text{ (mod 4)}$ with $\eta' \in \{\pm1\},$ then
\[
\Q(\PSL^{\epsilon}_2(q))=
\begin{cases} 
      \Q(\xi_{(q-1)/2}+\xi_{(q-1)/2}^{-1},\xi_{(q+1)/2}+\xi_{(q+1)/2}^{-1},\sqrt{\eta' q}) & q\text{ odd} \\
    \Q(\xi_{q-1}+\xi_{q-1}^{-1},\xi_{(q+1)}+\xi_{(q+1)}^{-1}) & q\text{ even.} 
   \end{cases}
\]
\end{theorem}

\begin{proof}
    Let $G = \SL^{\e}_n(q),\tilde{G} = \GL^{\e}_n(q)$ and $S = \PSL^{\e}_n(q)$. The statement for the $n=2$ case can easily be verified from the known generic character tables, so we need only consider the  case where $n > 2$ and we begin by showing 
   \[\Q(\PSL^{\epsilon}_n(q)) 
   \supseteq  \begin{cases} 
      \Q(\GL^{\epsilon}_{n-1}(q))(K,\sqrt{\eta q}) & q \text{ odd, } q \text{ not a square, } 2\le n_2\le (q-\epsilon)_2\\
      \Q(\GL^{\epsilon}_{n-1}(q))(K) & \text{otherwise}. \\
   \end{cases}
\]  If we assume that $q$ is odd, $q$ is not a square, and $2 \le n_2 \le (q-\epsilon)_2$, then,  using \cite[Theorem 6.1 and Remark 2.2]{SV19} and Corollary \ref{center}, if there exists a semisimple element $s 
\in \GL^{\epsilon}_n(q)$ such that $\operatorname{diag}(\zeta_{n_2})s$  is conjugate to $s,$ $s \in G,$ and $\chi\in \mathcal{E}(G,s)$, then $\sqrt{\eta q}$ lies in the field of values of the deflation of a constituent of $\chi|_G$ and $\sqrt{\eta q} \in \Q(S)$. Write $n_2 = 2^k$. Then $2^{k+1}$ divides $q^2-1$, if $s$ is the semisimple elements whose eigenvalues are all of the primitive $2^{k+1}$th roots of unity, all with equal multiplicity, then $s$ meets all the described criteria.

    We now show that  $\xi_{q-\epsilon} \in \Q(S)$ (note that the cases where $q-\e < 3$ are trivial so we my assume this is not the case). First consider the semisimple elements $s,t$ in $\tilde{G}$ with non-identity eigenvalues $(\zeta_{q-\epsilon},\zeta_{q-\epsilon}^{-1})$ and $(\zeta_{q-\epsilon},\zeta_{q-\epsilon},\zeta_{q-\epsilon}^{-2})$ respectively. Let $\pi:\GL^{\e}_n(q) \rightarrow \PGL^{\e}_n(q)$ be the canonical projection map and let $\chi_{\pi(s)} \in \mathcal{E}(G,\pi(s))$, $\chi_s \in \mathcal{E}(\tilde{G},s)$ and $\chi_{\pi(t)} \in \mathcal{E}(G,\pi(t))$ be semisimple characters. By Corollary \ref{center} these characters are trivial on the center, so it is sufficient  to show $\xi_{q-\epsilon} \in \Q(\chi_{\pi(s)},\chi_{\pi(t)})$.  To show this we first want to show $\xi_{q-\epsilon} + \xi_{q-\epsilon}^{-1}\in \Q(\chi_{\pi(s)})  $.  Since $\chi_{\pi(s)}$ is a constituent of $(\chi_s)_G$, it is immediate that $(\Q(\chi_s)_G) \subseteq \Q(\chi_{\pi(s)})$.  Using Remark \ref{tftgt} it is immediate that $\xi_{q-\epsilon} + \xi_{q-\epsilon}^{-1}\in \Q(\chi_s),$  so it is sufficient to show that $\Q(\chi_s) \subseteq \Q((\chi_s)_G)$ or equivalently to show that any Galois automorphism which fixes $(\chi_s)_G$ must also fix $\chi_s$. Let $\sigma \in \G$ such that $\sigma$ fixes $(\chi_s)_G$ up to conjugation. Then $\sigma$ must fix the series $\mathcal{E}(G,\pi(s)),$ since all the constituents of $\chi_G$ lie in this series. So in particular $\sigma$ fixes $\pi(s)$ up to conjugation. Thus there exists  $z= \operatorname{diag}(a,...,a) \in \mathbf{Z}(\GL_n(q))$, with $a \in  \GL^\e_1(q)$, such that $\sigma(s)z $ is conjugate to $s$. Then, since $n > 2$ we have $1$ is an eigenvalue of $s$ and it follows that $a \in \{\zeta_{q-\epsilon},\zeta_{q-\epsilon}^{-1},1\}$.  We  also see that  $a\sigma(\zeta_{q-\epsilon}),$ $a\sigma(\zeta_{q-\epsilon})^{-1},$ or $a$ must equal $1$. These two statements together imply that $\sigma(\zeta_{q-\epsilon}) \in \{\zeta_{q-\epsilon},\zeta_{q-\epsilon}^{-1}\}$ and $\sigma(s)$ is conjugate to $s$. Thus, $\sigma$ fixes $\chi_s$ as desired.

  We now consider $\Q(\chi_{\pi(t)})$. If $\sigma \in \G$ such that $\sigma(\xi_{q-\epsilon}) = \xi_{q-\epsilon}^{-1}$, then clearly $\sigma(\pi(s)) = \pi(s)$  and a straightforward computation similar to the computation in the preceding paragraph gives that, since that $S$ is not one of the listed exceptions,  $\sigma(\pi(t) )\neq \pi(t)$ up to conjugation. Then, $\Q(\chi_{\pi(t)})$ is not contained in $\Q(\xi_{q-\e}+ \xi_{q-\e}^{-1})$. Since $\Q(\xi_{q-\e}+\xi_{q-\e}^{-1})$ is an index 2 subfield of $\Q(\xi_{q-\epsilon})$ by \ref{NumberTheoryLemma}, this implies $\Q(\xi_{q-\e}) \subseteq \Q(\chi_{\pi(s)},\chi_{\pi(t)})\subseteq  \Q(S)$.

    We now want to show that $\Q(\GL^\e_{(n-3)}(q)) \subseteq \Q(\PSL^\e_n(q))$ when $n > 3$. The case of $n = 4 $ follows from the above so assume $n \ge 5$. Consider semisimple elements $s_2,...,s_{(n-3)} \in \GL^\e_n(q)$ where  the complete collection of  non-identity eigenvalues of $s_i$  is $\zeta_{q^i-\e^i},\zeta_{q^{i}-\e^i}^{\e q},...,\zeta_{q^i-\e^i}^{(\e q)^{i-1}},$ $ \zeta_{q^i-\e^i}^{-(1+\e q+...+(\e q)^{i-1})}$. Note that each $s_i$ has determinant 1.  Let $\pi$ be as in the above paragraph and let  $\chi_i \in \mathcal{E}(\tilde G,s_i)$ be a semisimple character for each $2\le i \le n-3$. Note that Corollary \ref{center} gives that each of these characters is trivial on the center. It is then sufficient to show $\Q(\xi_{q-\e},(\chi_2)_G,...,(\chi_n)_G) \supseteq \Q(\GL^{\e}_{n-3}(q))$. From the fundamental theorem of Galois theory it is sufficient to show that for all $\sigma \in \G$, if $\sigma$ does not fix $\Q(\GL^{\e}_{n-3}(q))$, then $\sigma$ does not fix one of the elements of   $\{\xi_{q-\epsilon},(\chi_2)_G,...( \chi_n)_G\}$. Further, by above, we need only consider  the case where $\sigma \in \Stab_\G(\xi_{q-\epsilon})$. We see that for each $s_i$ the eigenvalue $1$ has multiplicity at least $2$ and the eigenvalue $\zeta_{q^i-\e^i}^{-(1+\e q+...+(\e q)^{i-1})}  = \zeta^{-1}_{q-\e}$ has multiplicity $1$. These are the only $2$ eigenvalues which lie in the subgroup of  $\mathbb{F}_{q^2}^\times$ of order $q-\e $. Thus, if $\sigma \in \Stab_\mathcal{G}(\xi_{q-\e})$ and  $z:= \operatorname{diag}(a,...,a) \in\mathbf{Z}(\tilde G)$, then the only eigenvalues of $\sigma(s_i)z$ which lie in $\Q(\xi_{q-\epsilon})$ will be $a\zeta_{q-\epsilon}^{-1}$ with multiplicity $1$ and $a$ with multiplicity  at least $2$. It follows that $\sigma(s_i)z$ cannot be conjugate to $s_i$ unless $z = 1$ and $\sigma(s_i)$ is conjugate to $s_i$. If $\sigma \in \Stab_\mathcal{G}(\xi_{q-\e})$  fixes $(\chi_{s_i})_G$, then it must fix $\E(G,\pi(s_i)),$ since all of its constituents lie in this series. Thus, $\sigma$ must fix $\pi(s_i)$ and by the above argument must fix $s_i$. The first paragraph of the proof of Theorem \ref{main1} argues that $\Q(\GL_{n-3}(q)) = \Q(\xi_{q-\e},\chi_2,...,\chi_n)$, so this gives that $\sigma$ fixes $\Q(\GL_{n-3}(q))$, which is equivalent to what was to be shown.

   Now we show $\Q(\GL_{(n-2)}) \subseteq \Q(\PSL_n(q))$.  Let $s$ be semisimple in $\GL_n(q)$ with non-identity eigenvalues $\zeta_{q^{n-2}-\e^{n-2}},\zeta_{q^{n-2}-\e^{n-2}}^{\e q},...,\zeta_{q^{n-2}-\e^{n-2}}^{(\e q)^{n-3}}, \zeta_{q^{n-2}-\e^{n-2}}^{-(1+\e q+...+(\e q)^{n-3})}$, and let $\chi \in \mathcal{E}(\GL_n^\e(q),s)$ be semisimple.  Then we see that the only eigenvalues of $s$ that lie in  the subgroup of  $\mathbb{F}_{q^2}^\times$ of order $q-\e $ are $1$ and $\zeta_{q^{n-2}-\e^{n-2}}^{-(1+\e q+...+(\e q)^{n-3})} = \zeta_{q-\e}^{-1} $, each with multiplicity $1$.  We see that for $\sigma \in \operatorname{Stab}_\G(\zeta_{q-\e})$ we have that $\sigma(s_i) = zs_i$ for $z = \operatorname{diag}(a,a,...a) , a\neq 1$ if and only if $a =  \zeta_{q-\e}^{-1}$  and $a \sigma(\zeta_{q-\e}^{-1})=  \zeta_{q-\e}^{-1}a = a^2 = 1$. Thus, if $\zeta_{q-\e}^{-1}\neq -1$ we see that $\pi(s)$ being conjugate to $ \sigma(\pi(s))$ implies that $s$ is conjugate to $\sigma(s).$ Then we can argue as in the above case to obtain the result. Thus, the only case that we need to consider is when $(q,\e) = (3,1)$. Since we have assumed $S \neq \PSL_4(3)$ and $\PSL_3(3)$, the previous paragraph allows us to assume that  $\sigma$ acts trivially on $\Q(\xi_8+\xi_8^{3})$, since we already know $\xi_8 + \xi_8^{-3} \in \GL_{n-3}(q) \in \Q(S)$. We then take the semisimple element in $\GL_n(q)$ with eigenvalues 
$\zeta_{q^{n-2}-1},\zeta_{q^{n-2}-1}^q,...,\zeta_{q^{n-2}-1}^{q^{n-3}}, \zeta_8,\zeta_8^3$. Then a similar argument to the case of $q \neq 3$ gives the claim.

   Now we show  $\Q(\GL^{\e}_{(n-1)}) \subseteq \Q(\PSL^{\e}_n(q))$. We consider semisimple element $s \in \GL^{\e}_n(q)$ with non-identity eigenvalues $\zeta_{q^{n-1}-\e^{n-1}},\zeta_{q^{n-1}-\e^{n-1}}^{\e q},...,\zeta_{ q^{n-1}-\e}^{(\e q)^{n-1}}, \zeta_{q^{n-1}-\e^{n-1}}^{-(1+\e q+...+(\e q)^{n-2})}$. Then we see that the only eigenvalue of $s$ that lies in the subgroup of  $\mathbb{F}_{q^2}^\times$ of order $q-\e $  is $\zeta_{q^{n-1}-\e^{n-1}}^{-(1+\e q+...+\e q^{n-3})} = \zeta_{q-\e}^{-1}$. Thus, we see that for $\sigma \in \operatorname{Stab}_\G(\zeta_{q-\e})$ we have that $\sigma(s_i) \neq zs_i$ up to conjugation for $z = \operatorname{diag}(a,a,...a) $ with $a \neq 1$. Thus, $\pi(s)$ being conjugate to $ \sigma(\pi(s))$ implies that $s$ is conjugate to $\sigma(s).$ Then arguing as exactly as in the argument used to show $\Q(\GL^\e_{(n-3)}(q)) \subseteq \Q(\PSL^\e_n(q))$ gives the claim.

To finish the inclusion we need to show that $K \in \Q(S)$. Let $a = (q^n-\e^n)/(q-\e)$ and let $s$ be a semisimple element with eigenvalues $\zeta_a,\zeta_a^{\e q},...,\zeta_a^{(\e q)^{n-1}}$. Then, if $\chi \in \mathcal{E}(\tilde G,s)$ is semisimple, then Lemma \ref{neweigenvalues} and Corollary \ref{center} imply $K \in \Q(\chi_G)\subseteq \Q(S)$.

Now we show the other inclusion. Arguing exactly as in the last paragraph of the proof of Proposition \ref{SL} shows that it is sufficient to consider the deflations of characters in series indexed by the images in $\PGL_n^\e(q)$ of semisimple elements of $s \in\tilde{G}$ with  eigenvalues of the form $(\zeta,\zeta^{\epsilon q},...,\zeta^{(\epsilon q)^{n-1}})$ such that $|\zeta|=c$, where  $c|(q^n-\epsilon^n)$  and $c\nmid (q^k-\e^k)$ for all $1\le k <n$.
 Using \ref{center} we need only consider those with $s \in \SL_n(q)$. This is true if and only if $c |(q^n-\e^n)/(q-\e)$. Then Lemma using \ref{neweigenvalues} and arguing as in the proof of Proposition \ref{SL} gives that adjoining the restrictions of all such characters to $\Q$ gives $\Q(K)$ which gives the containment.
\end{proof}

Propositions \ref{SL} and \ref{PSL} give Theorem \ref{main}. 
        
    \end{section}

\section{Small degree extensions}
In this section we classify all degree-$2$ and degree-$3$ extensions of $q$ which occur as $\Q(S)$ for some simple group $S$.

\begin{proposition}\label{Alternating}
    The only degree-$2$ or degree-$3$ extensions of $\Q$ of the form $\Q(A_n)$ for $n \ge 5 $ are $\Q(\sqrt{5}),\Q(\sqrt{-7}),\Q(\sqrt{-15})$ and $\Q(\sqrt{21}).$
\end{proposition}
\begin{proof}
From the main theorem of \cite{RT95} we have that for $n > 24$ 
$$
\Q(A_n) = \Q(\sqrt{p^*} : p \text{ an odd prime, }  p \le n , p \neq n-2) 
$$ where $p^* = -p$ if $p  \equiv 3 ( \operatorname{mod} 4) $  and $p^{*} = p $ otherwise. So in particular for $n > 24$ we have $\Q(\sqrt{-3},\sqrt{5}) \subseteq \Q(A_n)$, so we need not consider these cases.

Character tables for $A_n$ with $n\le 14$ can be found in \cite{GAP}, which gives all the stated fields.  So it is sufficient so show that  $3 < |\Q(A_n):\Q| $ for  $15 \le n \le  24 $. Using \cite[Proposition 2.5.13]{JK} we see that for every partition $n = h_1+...+h_k$ such that each $h_i$ is odd and $h_1<h_2...<h_k$ we can find the value

$$
\frac{(-1)^{(n-k)/2} + \sqrt{(-1)^{(n-k)/2}\cdot h_1\cdot h_2\cdot...\cdot h_k}}{2}
$$
in the character table of $A_n$. If $n$ is odd then the partitions $(n)$ and $(1,3,n-4)$ correspond to values which give a degree-$4$ extension. If $n$ is even and not equal to 24 then  $(1,(n-1))$ and $(1,3,5, (n-9))$ are sufficient. If $n=24$ we take $(1,23)$ and $(1,3,7,13)$.
    \end{proof}

    Next we consider $\GL^{\epsilon}_n(q)$, which will be useful for considering simple groups.

    \begin{proposition}\label{GL23}
    Let $q$ be a power of a prime and $n\ge 2$ an integer. Then the following hold:
    \begin{enumerate}
   \item The only degree-$2$ or degree-$3$ extensions of $\Q$ of the form $\Q(\GL^{\e}_n(q))$ are $\Q(\sqrt{-7}),\Q(\sqrt{-3}),$ and $\Q(\sqrt{-2})$ 
   \item $|\Q(\GL_n^\e(q)):\Q| \le 3$ if and only if $(q,n) \in \{(2,2),(2,3),(3,2) \}$ or $\e = -1$ and $(n,q) = (4,2)$.
    \end{enumerate}
    \end{proposition}
\begin{proof}
First we show $(1)$. By Theorem \ref{main1} we have
    in all cases that $\Q(\xi_{q-e}) \subseteq \Q(\GL_n(q))$, so we need not consider cases where $\varphi(q-\e)  = |\Q(\xi_{q-\epsilon}):\Q|\ge 4$. Then using that fact that $\phi(n) \ge \sqrt{n/2}$ leaves only finitely many cases to check. We then see that  the only possibilities are $q = 2,3,4,5$ and $7$ when $\e = 1$ and $q = 2,3$ or $5$ when $\epsilon = -1$. Further in all cases we have $\Q(\xi_{q^2-1} + \xi_{q^2-1}^{\epsilon q} ) \subseteq  \Q(\GL^{\epsilon}_n(q))$ by Theorem \ref{main1}.   Applying Lemma \ref{NumberTheoryLemma} it is easy to show that  only cases in which we could possibly have $|\Q(\xi_{q^2-1} + \xi_{q^2-1}^{\epsilon q} , \xi_{q-\e}):\Q| < 3$ are when $q = 2$ or $q =3$. Note that, using the above argument, we have $\Q(\GU_2(2))  = \Q(\sqrt{-3})$ and $\Q(\sqrt{-3}) \subseteq \Q(\GU_n(2))$ for $n >3$.  Since any proper extension $\Q(\sqrt{-3})$ will have index at least $4$ over $\Q$ so  we need not consider these cases (note that  if  $n = 3$ we get $\Q(\GL_n(q)) = \Q(\sqrt{-3},\xi_9 + \xi_9^{7} + \xi_9^{4}) $ and $ \xi_9 + \xi_9^{7} + \xi_9^{4} = \xi_9(1 +\xi_3 + \xi_3^2) = 0,$ so $\Q(\GL_3(2))$ is also equal to $\Q(\sqrt{-3})$). Similarly, since $\Q(\xi_8 +\xi_8^{-3}) = \Q$, we see that $\Q(\GU_2(3)) = \Q(\sqrt{-1})$ and $\Q(\sqrt{-1}) \subseteq \Q(\GU_n(3))$ for $n >3$, so  we need not consider these cases. Again, arguing as above, we have $\Q(\GL_2(2)) = \Q$, so assume $n\ge 3$. Then $\Q(\xi_7 +\xi_7^2 + \xi_7^4) = \Q(\sqrt{-7}) \subseteq \Q(\GL^{\epsilon}_n(2)) $ and in particular $\Q(\GL_3(2)) = \Q(\sqrt{-7})$. Lastly, we see that $\Q(\GL_2(3)) = \Q(\xi_8+\xi_8^3) =\Q(\sqrt{-2}) \subseteq \Q(\GL_n(3))$. This gives $(1)$.

    Now we show $(2)$. By the proceeding paragraph we see that $|\Q(\GL_n^\e(q)):\Q|>3$ when $q\notin \{2,3\}$, so we need only consider the case  where $q \in \{2,3\}$. The previous paragraph also gives that if $(q,n) \in \{(2,2),(2,3),(3,2)\}$, then $|\Q(\GL_n^\e(q):\Q|\le 3,$.  First let $q = 3$. Then we may assume $n \ge 3$. If $\e = 1$ we have $\Q(\xi_{13}+\xi_{13}^3+\xi_{13}^9,\sqrt{-2}) \in \Q(\GL^{\e}_n(q)) $, and it is immediate that this field has sufficiently large degree over $\Q$.  If instead $\e = -1$ we have $\Q(\sqrt{-1},\xi_{28}+ \xi_{28}^{25} + \xi_{28}^{9}) \subseteq \Q(\GL_3^{\e}(3))$ and we claim that $\xi_{28}+ \xi_{28}^{25} + \xi_{28}^{19}$ has orbit of size $4$ under $\G$.  This follows from Lemma \ref{numbertheorylemmabig}.

    Now for the case where $q = 2$ and $n \ge 4$ we see that, if $\epsilon = 1$, then $\xi_{15} + \xi_{15}^4 + \xi_{15}^2 + \chi_{15}^8$ is contained in $\Q(\GL_n(q))$. We note that 
    $$
    \xi_{15} + \xi_{15}^4 + \xi_{15}^2 + \xi_{15}^8 = \xi_3(\xi_5 + \xi_5^4)+ \xi_3^2(\xi_5^2 + \xi_5^3 ) = 2\sqrt{-15}.
    $$ Thus, since $\sqrt{-7} \in \Q(\GL_n(q))$ as well, we see $|\Q(\GL_n(q)):\Q| \ge 4$. Lastly we consider when $\epsilon = -1$. Since $\xi_{15} + \xi_{15}^{-2} + \xi_{15}^4 + \chi_{15}^{-8} = -\sqrt{-3}$  we see that $\Q(\GU_4(2)) = \Q(\sqrt{-3})$. If $n>4$ we get  $\xi_{33}+\xi_{33}^{-2}+\xi_{33}^{4}+\xi_{33}^{-8}+\xi_{33}^{16}$ is contained in $\Q(\GU_n(q))$ and $|\Q(\GL_n(q)):\Q| \ge |\Q(\xi_{33}+\xi_{33}^{-2}+\xi_{33}^{4}+\xi_{33}^{-8}+\xi_{33}^{16}):\Q| =  4$ by Lemma \ref{numbertheorylemmabig}. 
\end{proof}

\begin{proposition}\label{PSLsmallfields}
    Let $q$  be a power  of a prime and $n \ge 2$. Then the only degree $2$ or $3$ extensions of $\Q$ of the form  $\Q(\PSL^{\epsilon}_n(q)) $ are $\Q(\sqrt{5}),\Q(\sqrt{-3})$ and $\Q(\sqrt{-7}).$
\end{proposition}

\begin{proof} Let $S = \PSL_n^\e(q)$.
Note that the exceptions listed in the statement of Proposition \ref{PSL} all have character tables which can be found in the character table libraries in \cite{GAP} and in all of those cases we have a field extension of degree at least $4$ of $\Q$, so we may assume that $S$ is not one of these groups. 
     Assume that $n \ge 3$. Then by Proposition \ref{PSL} we see that $\Q(\GL_{n-1}^{\epsilon}(q)) \subseteq \Q(\PSL_n(q))$, so by Proposition \ref{GL23} we need only check cases where $(q,n-1)  \in \{(2,2),(2,3),(3,2)
    \} $ or $(q,n-1,\e) = (2,4,-1)$. We then compute using Proposition \ref{PSL}  that in all cases we get an extension of degree at least $4$ over $\Q$ apart from the cases  $\Q(\PSL_3(2))$, $\Q(\operatorname{PSU}_3(2))$ and $\Q(\operatorname{PSU}_4(2))$ which have character fields $\Q(\sqrt{-7}),\Q$ and $\Q(\sqrt{-3})$ respectively. 

    Now assume $n =2 $. Then, by Proposition \ref{PSL} if we  let $\delta = (2,q)$, we have that that $\Q(\xi_{(q-1)/\delta}+\xi_{(q-1)/\delta}^{-1},\xi_{(q+1)/\delta}+\xi_{(q+1)/\delta}^{-1})\subseteq \Q(S)$. Further, Lemma \ref{NumberTheoryLemma}  gives us that $\Q(\xi_{(q-1)/\delta}+\xi_{(q-1)/\delta}^{-1}):\Q| = \varphi((q-1)/\delta)/2$ and $|\Q(\xi^{-1}_{(q+1)/\delta}+\xi_{(q+1)/\delta}^{-1}):\Q| = \varphi((q+1)/\delta)/2 $. Using this and known lower bounds on Euler's totient function reduces to cases which are available in  character table libraries in \cite{GAP}. We see that $\Q(\PSL_2(5)) = \Q(\sqrt{5})$ and all other degree-2 or degree-3 fields that occur are fields which have already been mentioned.
\end{proof}

\begin{proposition}\label{gapstuff2}
    Let $S$ be one of the simple groups $\PSp_4(4),$ $\PSp_6(3),$ $\operatorname{O}_7(3),$ $\operatorname{O}^{-}_8(2),$  $\operatorname{O}^{-}_{10}(2),$ $\operatorname{O}^{+}_{10}(2),$ $\operatorname{O}_{11}(2),$ $\operatorname{O}_{12}^+(2),$ $\operatorname{O}^{-}_{12}(2),$ $\operatorname{O}^{+}_{10}(2),$ or $\operatorname{O}_{13}(2)$, then $|\Q(S):\Q| \ge 4$. Further, we have the following:
    \begin{itemize}
    \item $ \Q(\PSp_6(2)) =\Q\operatorname{(O}_8^+(2))=  \Q,$
    \item $\Q(\PSp_8(2)) = \Q(\sqrt{17}),$ and
    \item $\Q(\operatorname{O}_8^+(3)) = \Q(\sqrt{13})$
    \end{itemize}

    \begin{proof}
        This follows from computation in \cite{GAP}.
    \end{proof}
    
\end{proposition}



We now prove Theorem \ref{SmallFields}.
\begin{proof}
Let $S$ be a non-abelian simple group. 
 Note that the case where $S$ is an alternating group was considered in \ref{Alternating}, and we can verify that $\Q(J_2) = \Q(\sqrt{5}),$ $\Q(M_{12}) = \Q(\sqrt{-11})$ and all other sporadics groups do not have character fields of degree $2$ or of degree $3$ by inspection of their character tables that can be found in \cite{GAP}, so we need only consider the case where $S$ is a group of Lie type. 
    Using \cite[Table 1]{Gec03} we see that for $S= \operatorname{E}_8(q)$ or $S = \operatorname{F}_4(q)$, the values of cuspidal unipotent characters give a sufficiently large extension, and if $S$ is of one the the types   $\operatorname{E_6},\operatorname{\,^2E_6}$ or $\operatorname{G}_2$ we see that $\sqrt{-3} \in \Q(S)$. Since Proposition \ref{PSLsmallfields} gives that $\Q(\sqrt{-3}) = \Q(\PSU_4(2))$ we need not consider these cases.  If $S = \operatorname{E_7}(q)$ for some $q$ we again use \cite[Table 1]{Gec03} to see that $\Q(\sqrt{-q}) \subseteq \Q(S)$. Then, \cite[Theorem 1.2]{Tre17}  gives that $S$ cannot have a character field of degree $2$ over $\Q$, so it must have degree at least $4$ and we need not consider it here. If $S$ is of one of the types $ \,^3\operatorname{D}_4$,$^2\operatorname{B}_2,$ or $^2\operatorname{G}_2$  we see via inspection of the known generic character tables that $|\Q(S):\Q| \ge 4$ in all cases. If $S$ is of type $ ^2\operatorname{F_2}$ then the descriptions of unipotent characters in \cite{Mal90} give that $|\Q(S):\Q| \ge 4$ in all cases.

    Now assume $S$ is a classical simple group of Lie type.   By Proposition \ref{PSLsmallfields} we need not consider the cases where $S$ is a group of type $\operatorname{A}$ or $\operatorname{^2A}$. We also need not consider the groups listed in the statement of \ref{gapstuff2}.  Then, we see that \cite[Theorem 1.2]{Tre17} gives that the remaining groups cannot be composition factors of groups with multi-quadratic fields of character values. Therefore, in all remaining cases it is sufficient to find a subfield of $\Q(S)$ degree $2$ over $\Q$.
    Let $G$ be one of the groups $\operatorname{Sp}_{2n}(q)$ for $n \ge 2$, $\operatorname{SO}_{2n+1}(q)$ for $n \ge 3$ or $\operatorname{SO}_{2n}^\e(q)$ for $n \ge 4$ and $\e \in \{\pm\}$. Let $\delta = 2$ in the case that $q $ is odd and $\delta = 1$ otherwise. In all cases we can find an element $s$ of the dual group of $G$ such that the only non-identity eigenvalues of $s$  are $\zeta_{q-1}^\delta$ and $\zeta_{q-1}^{-\delta}$, each with multiplicity $2$. We can also find a semisimple element $t$ such that  the only non-identity  eigenvalues of $t$ are   $\zeta_{q+1}^\delta$ and $\zeta_{q+1}^{-\delta}$ each with multiplicity $2$. We claim that for $\chi_s\in  \mathcal{E}(G,s)$ and $\chi_t\in  \mathcal{E}(G,s)$ we have $\xi_{(q-1)/\delta}+\xi_{(q-1)/\delta}^{-1} \in \Q(\chi_s)$ and $\xi_{(q+1)/\delta}+\xi_{(q+1)/\delta}^{-1} \in \Q(\chi_t)$. We see that $\xi_{(q-1)/\delta}+\xi_{(q-1)/\delta}^{-1} \in \Q(\chi_s)$ if and only if $\Stab_\G(\xi_{q-1)/\delta} +\xi_{(q-1)/\delta}^{-1} )  \ge \Stab_\G(\chi_s) $. If $\sigma \in \G$ does not fix $\xi_{q-1)/\delta} +\xi_{(q+1)/\delta}^{-1}$, it is immediate that $\sigma $ does not fix the conjugacy class of $s$ and $\sigma $ cannot fix $\chi_s$. Thus, $\xi_{(q-1)/\delta}+\xi_{(q-1)/\delta}^{-1} \in \Q(\chi_s)$ and the fact that $\xi_{(q+1)/\delta}+\xi_{(q+1)/\delta}^{-1} \in \Q(\chi_t)$ follows by a similar argument.  Proposition  \ref{stthings} implies that in all these cases we have  the exponent of $G^*/\mathbf{O}^{p'}(G^*)  $ is $\delta$. So if $\delta =1$ it is immediate that these character are trivial on the center. If  $ \delta =2$ we have that $s$ and $t$ lie in $\mathbf{O}^{p'}(G^*),$ since they are both squares in $G^*$ and by Corollary \ref{center} we see that these characters are trivial on the centers.

    If $S = \operatorname{PSp}_4(q)$ we see that both $\chi_s$ and $\chi_t$ in $\Irr(\operatorname{Sp}_4(q))$ deflate to characters in $\Irr(\operatorname{PSp}_n(q))$. Using Lemma \ref{NumberTheoryLemma} we see that for all $q$ apart from $q \in \{2,3,5,7,8,13\}$ we get that one of these characters has a field of values of degree $2$ or degree at least $4$, and in the case of $q = 8$ we see that $\Q(\chi_s)$ and $\Q(\chi_t)$ give $2$ distinct degree-$3$ extensions, so $|\Q(S):\Q| > 3$. If $q=13,$ then we then consider a character in the series indexed by  the semisimple element with eigenvalues $\zeta_{170},\zeta_{170}^{13},\zeta_{170}^{-1},\zeta_{170}^{-13}$ and arguing as in the other cases we see that this character gives an extension with of degree greater than $3$. All remaining cases can be checked in \cite{GAP}.

    We now consider the case where $G \neq \operatorname{Sp}_4(q)$. Let $\delta$ be as above  then in all cases we can find an element  $s\in G^*$ with non-identity eigenvalues $\zeta_{(q^3-1)/\delta},\zeta_{(q^3-1)/\delta}^q,\zeta_{(q^3-1)/\delta}^{q^2},\zeta_{(q^3-1)/\delta}^{-1},\zeta_{(q^3-1)/\delta}^{-q}$ and $\zeta_{(q^3-1)/\delta}^{-q^2}$. Then arguing as in Remark \ref{tftgt} we see that any Galois automorphism that does not fix the multiset $\{ \xi_{(q^3-1)/\delta},\xi_{(q^3-1)/\delta}^q,\xi_{(q^3-1)/\delta}^{q^2},\xi_{(q^3-1)/\delta}^{-1},\xi_{(q^3-1)/\delta}^{-q}$ and $\zeta_{(q^3-1)/\delta}^{-q^2}\}$, will not fix any  $\chi_s \in \mathcal{E}(G,s)$. Then we get that via the fundamental theorem of Galois theory $|\Q(\chi_s):\Q| \ge \Q(\xi_{(q^3-1)/\delta})/6$. Further, Corollary \ref{center},  the fact that $s$ is a square when $\delta=2,$ and the fact that the exponent of $G^*/\mathbf{O}^{p'}(G^*)  $ is $\delta$ give that this character is trivial on the center. Thus, if $G$ is of the form $\operatorname{Sp}_{2n}(q)$ we see that it is sufficient to show that $\varphi((q^3-1)/2)/6 \notin \{1,3\}$ in all cases. Using the fact that $\varphi(n) \ge \sqrt{n/2}$ for all $n$ reduces this to checking finitely many $q$. This gives that the only situations to consider are when $q =2$.

    In the case that $G$ is a special orthogonal group, we have that the character $\chi_s$ deflates to a character of the corresponding projective special orthogonal group. Then, the field of values of the  restriction of this character to $S$ will contain an index $\delta$ subfield of $\Q(\chi_s)$, thus it is sufficient to show that $\varphi((q^3-1)/2)/6\delta \notin \{1,3\}$ in this case. Arguing as in the above paragraph gives that the only situations to consider are when $q \in \{2,4\}$. In the case where $q = 4$ we see that we can find a semisimple element of the dual group with non-identity eigenvalues $\zeta_{65},\zeta_{65}^4,\zeta_{65}^{16},\zeta_{65}^{-1},\zeta_{65}^{-4}$ and $\zeta_{65}^{-16}$. Arguing as above using this element in the place of $s$ gives the result. 

    Now we looks at the remaining cases, all of which have $ q = 2$. By Proposition \ref{gapstuff2} we see that $\Q(\operatorname{Sp}_6(2)) = \Q$ and $\Q(\operatorname{Sp}_8(2))  =\Q(\sqrt{17})$. If $S = \operatorname{Sp}_{2n}(2)$ with $n > 4$ we can find a semisimple element in $s \in S^*$ with non-identity eigenvalues  $\xi_{17}^{2^k}$ for $0 \le k \le 8 $ and we get that the semisimple character in there series $\mathcal{E}(S,s)$ will have field of values containing $\sqrt{17}$. If instead $S$ is a group in type $\operatorname{B_n}, \operatorname{D_n}$ or $\operatorname{^2D_n}$ we can assume that $n > 6$ by Proposition \ref{gapstuff2}. Then we take the semisimple element in $G^* = G$ with non-identity eigenvalues $\xi_{65}^{2^k}$ for $0\le k < 12$
    and arguing as in pervious cases we see that the restriction of the semisimple character in the series corresponding to this semisimple element to corresponding simple group will given  a character with  character field of degree divisible by $2$ over $\Q$.  
    \end{proof}

\end{document}